\documentclass[11pt]{amsart}
\usepackage[top=1.2in,bottom=1.2in,left=1.2in,right=1.2in]{geometry}
\usepackage[inline]{enumitem}
\usepackage{amsthm}

\usepackage{mathtools}

\usepackage{float}
\usepackage{amssymb}
\usepackage{xypic,pinlabel}
\usepackage{hyperref}
\hypersetup{
    colorlinks,
    citecolor=black,
    filecolor=black,
    linkcolor=black,
    urlcolor=black
}

\newcommand{\D}{\mathbb{D}}
\newcommand{\R}{\mathbb{R}}
\newcommand{\Z}{\mathbb{Z}}

\newcommand{\F}{\mathcal F}
\newcommand{\G}{\mathcal G}
\DeclareMathOperator{\PMF}{PMF}

\newcommand{\B}{B}
\newcommand{\PB}{P{\hspace{-.1em}}B}

\DeclareMathOperator{\Stab}{Stab}
\DeclareMathOperator{\Fix}{Fix}
\DeclareMathOperator{\Aut}{Aut}
\DeclareMathOperator{\Inn}{Inn}
\DeclareMathOperator{\Mod}{Mod}
\DeclareMathOperator{\Conf}{Conf}

\newcommand{\p}[1]{\bigskip \noindent \emph{#1}.}

\theoremstyle{plain}
\newtheorem{theorem}{Theorem}[section]
\newtheorem{proposition}[theorem]{Proposition}
\newtheorem{lemma}[theorem]{Lemma}
\newtheorem{corollary}[theorem]{Corollary}
\newtheorem{question}[theorem]{Question}

\title{Homomorphisms between braid groups}

\author{Lei Chen}

\author{Kevin Kordek}

\author{Dan Margalit}

\address{Lei Chen \\ 4176 Campus Dr 4414\\ College Park, Maryland, 20742 }
\email{chenlei1991919@gmail.com}

\address{Kevin Kordek \\ School of Mathematics\\ Georgia Institute of Technology \\ 686 Cherry St. \\ Atlanta, GA 30332}
\email{kevin.a.kordek@gmail.com}

\address{Dan Margalit \\ School of Mathematics\\ Georgia Institute of Technology \\ 686 Cherry St. \\ Atlanta, GA 30332}
\email{margalit@math.gatech.edu}

\thanks{This material is based upon work supported by the National Science Foundation under Grant Nos.\ DMS-1057874, DMS-1811941, DMS-2203431.  The first author was supported by a Sloan Research Fellowship.} 

\begin{document}
\maketitle

\vspace*{-2em}

\begin{abstract}
We give a complete classification of homomorphisms from the braid group on $n$ strands to the braid group on $2n$ strands when $n$ is at least 5.  We also classify endomorphisms of the braid group on 4 strands, as well as homomorphisms from the commutator subgroup of the braid group on $n$ strands to the braid group on $2n-5$ strands.   Our classifications suggest a recursive classification of homomorphisms between any braid groups.  We also give a simple, geometric proof of a theorem of Lin that highly constrains the holomorphic maps that may exist between spaces of monic, square-free polynomials of two given degrees.  
\end{abstract}


\section{Introduction}

Let $\B_n$ denote the braid group on $n$ strands.  A fundamental problem about these groups is to classify all homomorphisms $\B_n \to \B_m$ for various $n$ and $m$.  Work of Artin \cite{artinbp}, Lin \cite{linbp}, Dyer--Grossman \cite{dg}, Bell--Margalit \cite{bellmargalit}, and Castel \cite{castel} gives a complete classification for $n \geq 6$ and $m \leq n+1$.  We extend the classification to the case $n \geq 5$ and $m \leq 2n$.  For $m = 2n$ there are three new types of homomorphisms that do not arise when $m < 2n$.  

 \p{Standard homomorphisms} As usual we denote the standard generators for $\B_n$ by $\sigma_1,\dots, \sigma_{n-1}$.  We have the following standard homomorphisms $\B_n \to \B_{2n}$ (we compose braids right to left).
\medskip
\begin{enumerate}[leftmargin=*,itemsep=.5em]
\item \emph{Trivial:} \hspace*{16.37ex} $\sigma_i \mapsto 1$
\item \emph{Inclusion:} \hspace*{13.825ex} $\sigma_i \mapsto \sigma_i$ 
\item \emph{Diagonal inclusion:} \hspace*{4.175ex} $\sigma_i \mapsto \sigma_i\sigma_{n+i}$ 
\item \emph{Flip diagonal inclusion:} \hspace*{-.075ex} $\sigma_i \mapsto \sigma_i\sigma_{n+i}^{-1}$ 
\item \emph{$k$-twist cabling:} \hspace*{8.53ex} $\sigma_i \mapsto \sigma_{2i}\sigma_{2i-1}\sigma_{2i+1}\sigma_{2i}\sigma_{2i-1}^k$
\end{enumerate}
\medskip
There is one $k$-twist cabling map for each $k \in \Z$.  Figure~\ref{fig:cable} shows the image of $\sigma_i$ under the $k$-twist cabling map.  The first two maps also define homomorphisms $\B_n \to \B_m$ with $n\le m < 2n$.

Our main theorem, Theorem~\ref{thm:main} below, says for for $n\geq 5$, every homomorphism $\B_n \to \B_{2n}$ is equivalent to exactly one of these standard homomorphisms.  In order to state the theorem precisely, we need to define the equivalence relation, which in turn relies on two other notions, conjugation and transvection.

\p{Almost-conjugations} Let $G$ be a group, and let $\rho_0 : \B_n \to G$ be a homomorphism.  We say that another homomorphism $\rho_1 : \B_n \to G$ is conjugate to $\rho_0$ if there is an inner automorphism $A \in \Inn G$ so that $\rho_1 = A \circ \rho_0$.  Further, we say that $\rho_1$ is \emph{almost-conjugate} to $\rho_0$ if $\rho_1 = A \circ \rho_0$ with $A \in \Aut G$.  Let $G \to \Aut G$ be the natural map to the inner automorphism group $\Inn G \leqslant \Aut G$, so an element $g$ maps to conjugation by $g$.  The sense in which almost-conjugations are similar to conjugations is that the compositions
\[
\B_n \stackrel{\rho_0}{\to} G \to \Aut G \quad  \text{ and } \quad \B_n \stackrel{\rho_1}{\to} G \to \Aut G,
\]
are conjugate whenever $\rho_0$ and $\rho_1$ are almost conjugate.

\p{Transvections} Next, given a homomorphism $\rho : \B_n \to G$ as above, we say that a homomorphism $\B_n \to G$ is a \emph{transvection} of $\rho$ if its action on generators is given by
\[
\sigma_i \mapsto \rho(\sigma_i)\, t
\]
for some $t \in G$.  We denote the transvection by $\rho^t$.  We emphasize that the transvection is only defined if the given formula defines a homomorphism.  In this case, we refer to $t$ as a transvecting element for $\rho$.  We will prove that in Section~\ref{sec:tv} that $t$ is a transvecting element for $\rho$ if and only if $\rho(\sigma_i)t\rho(\sigma_i)^{-1}$ is independent of $i$, and further this is true if and only if $t$ lies in the centralizer of $\rho(\B_n')$ (Proposition~\ref{prop:tv}).  In particular, if $t$ centralizes $\rho(\B_n)$ then $t$ is a transvecting element; in this case we refer to $\rho^t$ as a central transvection of $\rho$.  If $\rho^t$ is a central transvection of $\rho$ then we have the formula
\[
\rho^t(g) = \rho(g)\, t^{L(g)}
\]
where $L : \B_n \to \Z$ is the length homomorphism (or abelianization map).  As we will explain in Section~\ref{sec:tv}, it is a consequence of our main theorem that the transvecting elements for a homomorphism $\rho : \B_n \to \B_{2n}$ are essentially the elements supported in subsurfaces of $\D_{2n}$ on which the action of $\B_n$ on $\D_{2n}$ (via $\rho$) is cyclic. 

We say two homomorphisms $\B_n \to G$ are \emph{equivalent} if one is almost-conjugate to a transvection of the other (we verify in Section~\ref{sec:tv} that this is an equivalence relation).  We say two homomorphisms are \emph{centrally equivalent} if the transvection is central.  A homomorphism $\B_n \to G$ is equivalent to the trivial map if and only if it has cyclic image. 

\p{Statement of the main theorem} We now arrive at our main result, which completely classifies homomorphisms $\B_n \to \B_{2n}$ for $n \geq 5$.  We address the case $n=4$ in Theorem~\ref{thm:castel4} in Section~\ref{sec:castel}.

\begin{theorem}\label{thm:main}
Let $n\geq 5$, and let $\rho : \B_n \to \B_{2n}$ be a homomorphism.  Then $\rho$ is equivalent to exactly one standard homomorphism $\rho_0$.  Moreover $\rho$ is centrally equivalent $\rho_0$.

\end{theorem} 

\begin{figure}
\labellist
\small\hair 2pt
\pinlabel $k$ at 163 172
\endlabellist
\includegraphics[scale=.5]{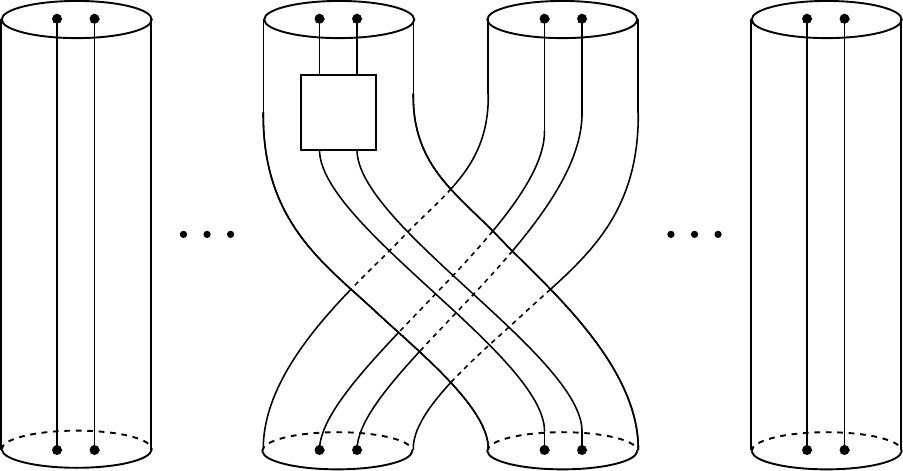}
\caption{The image of $\sigma_i$ under the $k$-twist cabling map: the $i$th and $(i+1)$st cable cross, and inside the $i$th cable there are $k$ half-twists}
\label{fig:cable}
\end{figure}

Of course Theorem~\ref{thm:main} is only as useful as our understanding of $\Aut(\B_n)$ and the central transvections.  It is a theorem of Dyer--Grossman that $\Aut(B_n)$ is isomorphic to $\B_n/Z(B_n) \rtimes \Z/2$, where the first factor acts by conjugation and the second factor acts by inversion: $\sigma_i \mapsto \sigma_i^{-1}$.  Further, the centralizers of the images of the standard homomorphisms $\B_n \to \B_{2n}$ can be completely described; see  Section~\ref{sec:tv}.  

\p{Maps to smaller braid groups} Theorem~\ref{thm:main} further gives a classification of all homomorphisms $\B_n \to \B_m$ with $m < 2n$. Indeed, given a map $\rho : \B_n \to \B_m$ with $m \leq 2n$, we may post-compose with the inclusion $\B_m\to \B_{2n}$ to obtain a homomorphism to which Theorem~\ref{thm:main} applies.  We say that $\rho$ is standard if the composition is.

\begin{corollary}
\label{cor:main}
Let $n\geq 5$, let $n \leq m < 2n$, and let $\rho : B_n \to B_{m}$ be a homomorphism.  Then $\rho$ is equivalent to exactly one standard homomorphism $\rho_0$, that is, the trivial map or the inclusion map.  Moreover $\rho$ is centrally equivalent to $\rho_0$.
\end{corollary}

\p{Sharpness of the lower bound}  The lower bound $n \geq 5$ in the statement of Theorem~\ref{thm:main} is sharp.  Indeed, there is a surjective homomorphism $\B_4 \to \B_3$ given by $\sigma_1,\sigma_3 \mapsto \sigma_1$ and $\sigma_2 \mapsto \sigma_2$; we refer to this as the \emph{exceptional homomorphism} $\B_4 \to \B_3$.  This map is not equivalent to any of the above standard homomorphisms.

There are also many non-standard homomorphisms $\B_3 \to \B_m$ for $m \geq 3$, for instance, $\sigma_1 \mapsto \sigma_1^{-1} \sigma_2^2$ and $\sigma_2 \mapsto \sigma_2^{-1}\sigma_1^2$.  This homomorphism was defined by Castel in a slightly different context \cite[Proposition 14.1]{castel_thesis}.  It is perhaps not surprising that such homomorphisms exist since the quotient of $B_3$ by its center is the amalgamated free product $\Z/4 *_{\Z/2} \Z/6$.

By post-composing the exceptional homomorphism $\B_4 \to \B_3$ with a non-standard homomorphism $\B_3 \to \B_m$, we obtain non-standard homomorphisms $\B_4 \to \B_m$ for $m \geq 3$.  

Despite the existence of the exceptional homomorphism $\B_4 \to \B_3$, we are still able to completely classify all homomorphisms $\B_4 \to \B_4$ (hence all homomorphisms $\B_4 \to \B_3$) by extending our methods to this case; see Theorem~\ref{thm:castel4} in Section~\ref{sec:castel}.  In light of this result, it would make sense to refer to the exceptional homomorphism $\B_4 \to \B_3$ as standard.

\p{Homomorphisms of the commutator subgroup}  The commutator subgroup $\B_n'$ of $\B_n$ is the kernel of the abelianization $L : \B_n \to \Z$.  There is an inclusion map $i : \B_{n-2} \to \B_n'$ given by $\sigma_i \mapsto \sigma_i \sigma_{n-1}^{-1}$ for $i \in \{1,\dots,n-3\}$.  A homomorphism 
$\rho : \B_n' \to \B_{2n-5}$ thus induces a homomorphism $\rho \circ i : \B_{n-2} \to \B_{2n-5}$.  By Theorem~\ref{thm:main}, the latter is equivalent to the trivial map or the inclusion map; this gives a strong restriction on $\rho$.    Also, since $\B_n'$ is perfect, all transvections of homomorphisms of $\B_n$ are trivial.  We thus arrive at the following corollary.

\begin{corollary}
\label{cor:bnp}
Let $n \geq 7$.  If $\rho : \B_n' \to \B_{2n-5}$ is a nontrivial homomorphism, $\rho$ is almost-conjugate to the inclusion map.
\end{corollary}

We derive Corollary~\ref{cor:bnp} from Theorem~\ref{thm:main} in Section~\ref{sec:cor}.  This corollary generalizes a result of the second and third authors \cite{bnprime}, who proved for $n \geq 7$ that any nontrivial homomorphism $\B_n' \to \B_n$ is almost-conjugate to the inclusion map.  The argument in that paper uses a different approach, the theory of totally symmetric sets.  Orevkov \cite{orevkov2} later extended the classification to homomorphisms $\B_n' \to \B_n$ with $n \geq 4$.

In a 1996 preprint, Vladmir Lin asked a series of four increasingly general questions \cite[0.9.2(b)--0.9.2(e)]{linpreprint} about endomorphisms of $\B_n'$, the last being: does every automorphism of $\B_n'$ extend to an automorphism of $\B_n$?  The theorem of the second and third authors already implies that the answer is yes, and Corollary~\ref{cor:bnp} is a further extension.

Prior to all of these results, the group of automorphisms of $\B_n'$ was determined for $n \geq 4$ by Orevkov \cite{orevkov}.  An alternate proof for $n \geq 7$ was given by McLeay \cite{mcleay}.

\p{Spaces of polynomials} Let $\textrm{Poly}_n$ denote the space of monic, square-free polynomials of degree $n$ with complex coefficients.  This space is the same as the space of unordered configurations of $n$ points in the plane (the $n$ points are the roots).  As such, we have that $\pi_1(\textrm{Poly}_n)$ is isomorphic to $\B_n$.

There is a surjective map $\textrm{Poly}_4 \to \textrm{Poly}_3$ that arises in the resolution of quartic polynomials into cubic polynomials.  The induced map on fundamental groups is the exceptional homomorphism $\B_4 \to \B_3$ described above.  

By work of Lin \cite[Theorem 9.4]{lin04} and Murasugi \cite{murasugi}, a holomorphic map $\textrm{Poly}_n \to \textrm{Poly}_m$ induces a homomorphism $\B_n \to \B_m$ that sends periodic elements to periodic elements (a braid is periodic if and only if it has a central power).  Lin refers to such a homomorphism as a \emph{special homomorphism}.  We prove the following result, which generalizes a theorem of Lin.

\begin{theorem}
\label{thm:lin}
Let $n \geq 5$ and $m \geq 1$.  If $m$ and $m-1$ are both indivisible by one of $n$ or $n-1$, then all special homomorphisms $\B_n \to \B_m$ have cyclic image.
\end{theorem}

Lin proved that if $n(n-1)$ does not divide $m(m-1)$, then all special homomorphisms $B_n \to B_m$ have cyclic image \cite[Corollary 1.16]{linbp}.  Since $\gcd(n,n-1)=\gcd(m,m-1)=1$, Lin's hypothesis is stronger than the one in Theorem~\ref{thm:lin}.

We give a simple, geometric proof of Theorem~\ref{thm:lin} in Section~\ref{sec:torsion}; see Proposition~\ref{prop:numthy}(1) and Lemma~\ref{lem:cycliccyclic}.  

Our results give further information about special homomorphisms (hence about maps $\textrm{Poly}_n \to \textrm{Poly}_m$).  For instance, Theorem~\ref{thm:main} implies for $n \geq 5$ that the only special homomorphisms $\B_n \to \B_n$ are equivalent to the identity or the trivial map (this part follows from Castel's work; see below).

Also, our classification of homomorphisms $\B_4 \to \B_4$ (Theorem~\ref{thm:castel4}) shows that the only special homomorphisms $\B_4 \to \B_4$ are equivalent to the trivial map, the identity map, or the composition $\B_4 \to \B_3 \to \B_4$, where the first map is the (surjective) exceptional homomorphism and the second map is the standard inclusion.  Finally, it follows from Theorem~\ref{thm:castel4} that the only special homomorphisms $\B_4 \to \B_3$ are equivalent to either the trivial map or to the exceptional map.  We can interpret the last statement as saying that there is essentially only one way to resolve quartic polynomials into cubic polynomials.

\p{Maps to larger braid groups} Because the number of different types of standard homomorphisms $\B_n \to \B_m$ jumps from 2 to 5 when $m$ increases from $2n-1$ to $2n$ (and the total number of standard homomorphisms jumps from 2 to $\infty$), it may seem hopeless to classify all homomorphisms $\B_n \to \B_m$ when $m$ is large.  But the problem is more tractable than it seems at first: each of the 5 standard homomorphisms $\B_n \to \B_{2n}$ is built from the standard homomorphisms $\B_n \to \B_m$ with $m \leq n$, in a sense which we now explain.

In this discussion, it will be useful to regard $\B_n$ as the mapping class group of a disk $\D_n$ with $n$ marked points $\{x_1,...,x_n\}$; this is the group of connected components of the group of homeomorphisms of $\D_n$ that fix the boundary pointwise.    A \emph{multicurve} in $\D_n$ is the isotopy class of a nonempty collection of pairwise disjoint, homotopically nontrivial, non-peripheral, homotopically distinct simple closed curves in $\D_n-\{x_1,...,x_n\}$.  A map $\rho : \B_n \to \B_m$ is \emph{reducible} if there is a multicurve $M$ in $\D_m$ preserved by $\rho(\B_n)$.  

We think of a reducible homomorphism as a cabling map, and in fact we will refer to them as such in what follows.  In the standard braid picture, the multicurve $M$ traces out the boundary of a cable.  We allow for the possibility that $M$ has a single component (so that there is only one cable with multiple braid strands), for instance in the case of the inclusion map $\B_n \to \B_{n+1}$.

Our Theorem~\ref{thm:main} can be thought of as saying that every map $\B_n \to \B_{2n}$ is a cabling, and further classifying all cablings.  For example, the inclusion map has one cable with $n$ strands inside.  The diagonal and flip diagonal inclusions have two cables, each with $n$ strands inside, and the $k$-twist cablings have $n$ cables, each with two strands inside.  

One can hope to classify homomorphisms $\B_n \to \B_m$ with $m > n$ in a similar manner.  As a first step, we have the following question.

\begin{question}
\label{q}
Let $m > n$.  Is it true that all homomorphisms $\rho : \B_n \to \B_m$ with non-cyclic image are cablings?  In other words, are they all reducible?
\end{question}

If the answer to this question is yes, then one can further hope to give a recursive classification of homomorphisms $\B_n \to \B_m$ as follows.  As explained in Section~\ref{sec:pkg}, a cabling map $\rho : \B_n \to \B_m$ decomposes into a collection of homomorphisms from $\B_n$ (or the subgroups of finite index $L^{-1}(k\Z)$) to braids groups $\B_{m_i}$ with $m_i < m$.  For example, the flip diagonal inclusion map $\B_n \to \B_{2n}$ decomposes into three maps: the inclusion map $\B_n \to \B_n$, the inversion map $\B_n \to \B_n$, and the trivial map $\B_n \to \B_2$ (the latter is the exterior component, or outer level, of the flip diagonal inclusion map, described in Section~\ref{sec:pkg}).  In this way, a classification of maps $\B_n \to \B_\ell$ with $\ell < m$  (and $L_n^{-1}(k\Z) \to \B_\ell$ with $\ell < m$) yields, recursively, a classification of maps $\B_n \to \B_m$. 

Question~\ref{q} is one version of the question of whether all homomorphisms between mapping class groups are induced---in Mirzakhani's words---from ``some manipulations of surfaces'' \cite[p. 3]{AS}.  There are many results showing that homomorphisms between certain mapping class groups are induced by inclusions of surfaces; see the work of Ivanov \cite{ivanov}, Irmak \cite{irmak,irmak2}, Shackleton \cite{shack}, and Aramayona--Souto \cite{AS}.  There are many other related results, most of which can be found in the references to Ivanov's paper on MathSciNet.

Aramayona--Leininger--Souto \cite{als} gave examples of injective homomorphisms between mapping class groups of closed surfaces.  Since it is impossible for one closed surface to embed into another, these homomorphisms are of a different nature than the ones discussed in the previous paragraph; in fact, in their examples there are pseudo-Anosov mapping classes that map to multitwists.  On the other hand, the Aramayona--Leininger--Souto homomorphisms are constructed by lifting to a covering space, which is indeed a manipulation of surfaces.   

One might expect that the Aramayona--Leininger--Souto construction could be used to produce homomorphisms $\B_n \to \B_m$ where pseudo-Anosov elements map to multitwists.  This seems unlikely, however, because many of the branched covers of $\D_n$ have positive genus.

\p{Prior results} An early precursor to Theorem~\ref{thm:main} is a theorem of Artin from 1947, which states that any homomorphism from $\B_n$ to the symmetric group $S_n$ with transitive image is either cyclic or is conjugate to the standard projection \cite{artinbp}.  In 1996, Lin \cite[Theorem~F]{linbp} classified all homomorphisms $\B_n \to S_{2n}$ with transitive image (they all arise from our standard homomorphisms $\B_n \to \B_{2n}$).  Using the classification of homomorphisms $\B_n \to S_{k}$ with $k < n$, he further proved \cite[Theorem A]{linbp} that every homomorphism $\B_n\to \B_k$ has cyclic image provided $k<n$. 

Dyer--Grossman \cite[Theorem 19]{dg} proved in 1981 that $\Aut(\B_n) \cong \B_n/Z(B_n) \rtimes \Z/2$, answering a question of Artin (Pietrowski and Solitar \cite{ps} previously proved this for $n \leq 4$ and conjectured the more general statement).  Bell and the third author proved in 2006 that every injective homomorphism $\B_n \to \B_{n+1}$ is equivalent to the standard inclusion \cite[Main Theorem 2]{bellmargalit}.  Castel improved on this result by showing that every homomorphism $\B_n \to \B_{n+1}$ with non-cyclic image is of the same form \cite[Theorem 4(iii)]{castel}.  

There are many other works describing various types of maps between various types of braid groups; see, for instance, the work of An \cite{an}, Bardakov \cite{bardakov}, Bell and the third author \cite{bm2}, Childers \cite{childers}, Cohen \cite{cohen}, Irmak--Ivanov--McCarthy \cite{iim}, Leininger and the third author \cite{lm}, McLeay \cite{mcleay}, Orevkov \cite{orevkov}, Zhang \cite{zhang}, and the first author \cite{chen}.

\p{Idea of the proof} The overarching strategy for the proof of Theorem~\ref{thm:main} is to classify homomorphisms $\B_n \to \B_m$ inductively.  The base case $m=n$ is due to Castel (as mentioned above, this is \cite[Theorem 4(iii)]{castel}).  We give a new, relatively short proof of Castel's theorem.  Also, we improve the hypothesis of Castel's theorem from $n \geq 6$ to $n \geq 5$.  As in the statement of Theorem~\ref{thm:main}, this bound is sharp because of the surjective homomorphism $\B_4 \to \B_3$.  

For both the base case and the inductive step, we analyze a homomorphism $\B_n \to \B_m$ by considering the various possibilities for the images of the periodic elements of $\B_n$.  The braid group is torsion free; the periodic elements of $\B_n$ are defined to be the ones that are periodic in the sense of the Nielsen--Thurston classification.  

In the arguments, we focus on a specific periodic element $\alpha_1 \in \B_n$; this corresponds to rotation of $\D_n$ by $2\pi/n$.  According to the Nielsen--Thurston classification theorem, there are three possibilities for the image of $\alpha_1$ under a homomorphism $\rho : \B_n \to \B_m$: it can be either pseudo-Anosov, periodic, or reducible.  We treat the three cases in turn.  We give a more detailed description of our approach in Section~\ref{sec:overview}.

Our approach in this paper stands in contrast to the works of Ivanov \cite{ivanov}, Ivanov--McCarthy \cite{im}, Bell and the third author \cite{bellmargalit}, and Aramayona--Souto \cite{AS}, where homomorphisms between mapping class groups are understood by considering the different possibilities for the image of a Dehn twist.  The approach of using periodic elements has, on the other hand, also been used in the theory of mapping class groups, for instance in the work of Harvey--Korkmaz \cite{HK}, of Tchangang \cite{tchangang}, and of Lanier and the first author \cite{chenlanier}.

\subsection*{Acknowledgments} The authors would like to thank Justin Lanier and Nick Salter for helpful comments and conversations.  We would also like to thank the participants of the PATCH seminar in Philadelphia for a lively discussion that helped with the writing.  We are especially grateful to Benson Farb for several helpful discussions, and for pointing out the connection between our work and the resolution of quartic polynomials.  Finally, the authors are grateful to two anonymous referees for many corrections and comments that greatly improved the paper.


\section{Setup and Overview}
\label{sec:overview}

We give here an overview of the paper, and also introduce some notation and ideas that will be used throughout.  In this paper we make extensive use of the Nielsen--Thurston classification for mapping class groups and the related theory of canonical reduction systems; see \cite[Chapter 13]{farbmargalit}.

\p{Outline} As per the introduction, we prove Theorem~\ref{thm:main} by classifying maps $\B_n \to \B_m$ inductively.  Along the way, we work with the group of even braids $\B_n^2$.  This is the subgroup of $\B_n$ generated by all squares of all elements, or alternatively the kernel of the mod 2 abelianization homomorphism $\B_n \to \Z/2$.  Our induction has three steps.

\medskip

\noindent \emph{Base case.} A classification of homomorphisms $\B_n \to \B_n$.

\smallskip

\noindent \emph{Extension of the base case.} A classification of homomorphisms $\B_n^2 \to \B_n$.

\smallskip

\noindent \emph{Inductive step.} A classification of homomorphisms $\B_n \to \B_m$ with $n < m \leq 2n$.  

\medskip

As above, the base case is a theorem of Castel.  As discussed in the introduction, for all three steps we analyze a homomorphism from $\B_n$ to $\B_m$ by considering the various possibilities for the images of the periodic elements of $\B_n$.  Again, an element is periodic if it is periodic in the sense of the Nielsen--Thurston classification.  Equivalently, an element of $\B_n$ is periodic if its image has finite order in $\bar \B_n$, the quotient of $\B_n$ by its center $Z(\B_n)$. 

The quotient $\bar \B_n$ is isomorphic to a subgroup of the mapping class group of a sphere $S_{0,n+1}$ with $n+1$ marked points, namely, the subgroup consisting of elements that fix a distinguished marked point $p$.  Given a homomorphism $\rho : \B_n \to \B_m$ we will often consider the associated homomorphism $\bar \rho : \B_n \to \bar \B_m$, which is the post-composition of $\rho$ with the projection $\B_m \to \bar \B_m$.  

We specifically utilize two particular periodic elements $\alpha_1 = \sigma_{n-1} \cdots \sigma_2 \sigma_1$
 and $\alpha_2 = \sigma_{n-1} \cdots \sigma_2\sigma_1^2$ (again composing braids right to left).  The elements  $\alpha_1^n$ and $\alpha_2^{n-1}$ are equal; we denote this element by $z$.  The element $z$ generates $Z(\B_n) \cong \Z$.  We denote by $\bar \alpha_k$ the image of $\alpha_k$ in $\bar \B_n$, so $\bar \alpha_1$ is a rotation of $S_{0,n+1}$ by $2\pi/n$ and $\bar \alpha_2$ is a rotation by $2\pi/(n-1)$.  

The remainder of the paper is organized into three parts.  The first part is setup, the second part is the proof of Castel's theorem (the base case), and the third part is the rest of the proof of Theorem~\ref{thm:main}, including the extension of the base case and the inductive step.

\p{Part I: Setup} In Sections~\ref{sec:pkg} and~\ref{sec:tv} we introduce two basic notions that are used throughout the paper.  More specifically, in Section~\ref{sec:pkg} we describe the stabilizer in the braid group of an unnested multicurve.  We give a semi-direct product structure on the latter that we call the interior/exterior decomposition.  There are two versions, Lemmas~\ref{lem:pkg} and~\ref{lem:pkg2}.  Then in Section~\ref{sec:tv} we prove that the relation from the introduction is an equivalence relation (Proposition~\ref{prop:eq}) and also prove the characterizations of transvecting elements from the introduction (Proposition~\ref{prop:tv}).  

\p{Part II: Proof of Castel's theorem} In Sections~\ref{sec:cake} through~\ref{sec:castel} we prove (and extend) Castel's theorem, which characterizes all homomorphisms $\B_n \to \B_n$.  The first three of these sections introduce some of the key tools that are used for both Castel's theorem and our inductive step.  First, in Section~\ref{sec:cake} we give a combinatorial-topological lemma about curves and rotations.  There are two versions, Propositions~\ref{prop:cake1} and~\ref{prop:cake2}.  Then in Section~\ref{sec:torsion} we state and prove Proposition~\ref{prop:numthy}, which gives number-theoretic restrictions on the image of a periodic element of $\B_n$ in $\bar \B_m$.  Finally in Section~\ref{sec:pA} we show in Proposition~\ref{prop:torsiontopA} that if $\rho(z)$ is pseudo-Anosov, then $\bar \rho$ has cyclic image.  

In Section~\ref{sec:castel} we combine the above tools with a result of Lin in order to prove our extension of Castel's theorem.  The argument proceeds as follows.  Let $\rho : \B_n \to \B_n$ be a homomorphism and let $\bar \rho$ be the associated homomorphism to $\B_n \to \bar \B_n$.  As in the introduction, there are three possibilities for $\bar \rho (\alpha_1)$ under a homomorphism $\rho : \B_n \to \B_m$: it can be either pseudo-Anosov, periodic, or reducible.  

When $\bar \rho(\alpha_1)$ is pseudo-Anosov, we show that $\rho$ has cyclic image.  The idea is that if $\bar \rho(\alpha_1)$ is pseudo-Anosov, then its power $\bar \rho(z)$ is also pseudo-Anosov.  Since $\bar \rho(\B_n)$ is contained in the centralizer of $\bar \rho(z)$, and since the centralizers of pseudo-Anosov mapping classes are completely understood (by work of McCarthy), we can conclude that $\bar \rho$ has abelian (hence cyclic) image, and hence $\rho$ has cyclic image.  The details of this argument are given in Section~\ref{sec:pA}.

When $\bar \rho(\alpha_1)$ is periodic we show that $\rho$ either has cyclic image or is equivalent to the identity homomorphism.  The argument proceeds as follows.  We show in Section~\ref{sec:torsion} that if $\bar \rho(\alpha_1)$ is periodic and $\rho$ has non-cyclic image then (up to replacing $\rho$ by an equivalent homomorphism) $\bar \rho(\alpha_1)$ generates $\langle \bar \alpha_1 \rangle$, that is, $\bar \rho(\alpha_1) =  \bar \alpha_1^k$ with $\gcd(k,n)=1$.  To give the idea, suppose that $\bar \rho(\alpha_1) = \bar \alpha_2$ (which certainly does not generate $\langle \bar \alpha_1 \rangle$).  In this case, 
\[
\bar \rho(\alpha_1^{n-1}) = \bar \alpha_2^{n-1} = 1.
\]
But the normal closure of $\alpha_1^{n-1}$ in $\B_n$ contains the commutator subgroup (this fact is an instance of the well-suited curve criterion of Lanier and the third author \cite{laniermargalit}).  It follows that $\bar \rho$, hence $\rho$, has cyclic image.  

Continuing with the case where $\bar \rho(\alpha_1)$ is periodic, we show that if $\rho$ has non-cyclic image, then $\bar \rho(\alpha_1)$ is equal to $\bar \alpha_1^{\pm 1}$.  To do this we consider the interaction between $\bar \rho(\alpha_1)$ and the canonical reduction system $M$ of $\bar \rho(\sigma_1)$.  Since $\alpha_1^k \sigma_1 \alpha_1^{-k}$ commutes with $\sigma_1$ for $2 \leq k \leq n-1$, it must be that the image of $M$ under $\bar \rho(\alpha_1)^k$ is disjoint from $M$ for such $k$.  The combinatorial topology lemma in Section~\ref{sec:cake} then implies $\bar \rho(\alpha_1) = \bar \alpha_1^{\pm 1}$.  We then leverage this equality  in order to conclude that $\rho$ is equivalent to the identity homomorphism.  

If $\rho(\alpha_1)$---equivalently, $\bar \rho(\alpha_1)$---is reducible, this means that $\rho(\alpha_1)$ preserves a multicurve $M$ in $\D_n$.  It follows that $\rho(z)$, hence all of $\rho(\B_n)$, preserves $M$.  It must be that $M$ has fewer than $n$ components, and each component contains fewer than $n$ marked points in its interior.  We can can apply the interior/exterior decomposition from Section~\ref{sec:pkg} to decompose $\rho(\B_n)$ into a semi-direct product of braid groups of smaller index, and use induction to show that $\rho$ has cyclic image.

After proving (our extension of) Castel's theorem we give the classification of homomorphisms $\B_4 \to \B_4$ at the end of Section~\ref{sec:castel2}.  The proof follows the same outline as the one we use for Castel's theorem.

\p{Part III: Proof of the main theorem} Sections~\ref{sec:castel2} through~\ref{sec:proof} comprise the proof of Theorem~\ref{thm:main}, starting from the base case.  As above, we begin in Section~\ref{sec:castel2} by proving our extension of the base case, Theorem~\ref{thm:castel2}, which is a classification of homomorphisms $\B_n^2 \to \B_n$.  The proof follows the same outline as the proof of Castel's theorem.  We remark that it is possible to derive Castel's theorem from the corresponding theorem for $\B_n^2$; we prove Castel's theorem first because it is simpler and is of more independent interest.

In Section~\ref{sec:cabling} we study 2-fold cabling maps.  These are maps $\rho : \B_n \to \B_{2n}$ where $\rho(\B_n)$ preserves a collection of $n$ disjoint curves, each surrounding two marked points, and where the induced action of $\B_n$ on these curves is (conjugate to) the standard map $B_n \to S_n$.  We prove in Proposition~\ref{prop:cabling} that every 2-fold cabling map is equivalent to one of the standard $k$-twist cabling maps.

We prove Theorem~\ref{thm:main} in Section~\ref{sec:proof} by giving the inductive step.  This step again follows the same approach as the proof of Castel's theorem.  Given $\rho : \B_n \to \B_m$ with $n < m \leq 2n$ we consider the three possibilities for $\bar \rho(\alpha_1)$.  If $\bar \rho(\alpha_1)$ is pseudo-Anosov or periodic then as before we conclude that $\rho$ has cyclic image.  In fact this case is easier than the case $m=n$ since we can more quickly rule out that $\bar \rho(\alpha_1)$ is a power of $\bar \alpha_1$.

The case where $\bar \rho(\alpha_1)$ is reducible is the most subtle part of the proof.  In this case we conclude as before that $\rho(\B_n)$ preserves a multicurve $M$.  There are then a number of subcases, according to the topological type of $M$ and the action (through $\rho$) of $\B_n$ on the set of components.  We may consider the latter action as a homomorphism $\B_n \to \Sigma_k$, where $k$ is the number of components of $M$. 

One subcase is where $M$ consists of $n$ curves surrounding two marked points each and the resulting homomorphism $\B_n \to \Sigma_n$ is the standard homomorphism (up to conjugacy).  In this case $m=2n$ and $\rho$ is a 2-fold cabling map, and so we may apply Proposition~\ref{prop:cabling}.

Another subcase is where $M$ consists of a single curve surrounding more than $n$ marked points; in this case we may restrict to the interior of $M$ and apply induction.  

The most difficult case, which only arises when $m=2n$, is where $M$ consists of two curves surrounding $n$ marked points each and where $\B_n$ acts nontrivially on the components of $M$; in this case $\B_n^2$ acts trivially on the components of $M$ and so we may apply Theorem~\ref{thm:castel2} to show that $\rho$ is equivalent to the diagonal embedding.  

We end the paper in Section~\ref{sec:cor} by proving Corollary~\ref{cor:bnp} as a consequence of Theorem~\ref{thm:main}.


\section{The interior/exterior decomposition}
\label{sec:pkg}

In this section we introduce a basic tool for handling cablings $\rho : \B_n \to \B_m$.  As in the introduction, the cabling maps are exactly the ones that are reducible in the sense that there is some multicurve $M$ preserved by $\rho(\B_n)$.  We first address the case where each component of $M$ is fixed, and then the more general case.

\subsection{Braids fixing a multicurve}\label{sec:fix} Let $M = \{ c_1, \dots, c_k \}$ be an un-nested multicurve in $\D_m$.  We say that $\D_m$ is \emph{standard} if it is a convex disk in the Euclidean plane and the $m$ marked points lie on a horizontal line.  We further say that $M$ is \emph{standard} if it lies in a standard $\D_m$ and all of the $c_i$ are convex curves in $\D_m$.  The reason for these definitions is that they give an (almost) canonical identification between the mapping class group description of the braid group and the description in terms of braid diagrams, and hence in terms of the standard presentation of the braid group.  Indeed, a mapping class of $\D_m$ induces an element of the fundamental group of the (unordered) configuration space $\Conf(\R^2,m)$.  And then by tracing out this loop in $\R^2 \times [0,1]$ and projecting to $(\R \times \{0\}) \times [0,1]$ (and keeping track of the crossing information) we obtain a braid diagram. 

For any un-nested multicurve $M$, let $\Fix_{\B_m}(M)$ denote the subgroup of $\B_m$ given by the intersection of the stabilizers of the $c_j$ for $i\in\{1,...,k\}$.  We will describe two homomorphisms $\Pi_i^M$ and $\Pi_e^M$, the \emph{interior and exterior maps}, with domain $\Fix_{\B_m}(M)$.  We first describe the maps for $M$ standard and then in the general case.  

Suppose now that $M$ is standard.  Say that each $c_j$ surrounds exactly $p_j$ marked points, and let $p = p_1 + \cdots + p_k$.   Let $\Delta_j$ denote the closed disk bounded by $c_j$, and let $\Delta = \cup \Delta_j$.  Let $\D_m^\Delta$ denote the disk obtained from $\D_m$ by collapsing each $\Delta_j$ to a point.  We consider the latter as a disk with $m-p+k$ marked points, with each marked point coming from a $\Delta_j$ or a marked point of $\D_m$ not in the interior of any $\Delta_j$.  In fact, $\D_m^\Delta$ can be naturally identified with a standard disk $\D_{m-p+k}$ because the $\Delta_j$ are convex. 

There are homomorphisms
\begin{align*}
\Pi_i^M &: \Fix_{\B_m}(M) \to \B_{p_1} \times \cdots \times \B_{p_k}, \text{ and} \\
\Pi_e^M &: \Fix_{\B_m}(M) \to \B_{m-p+k},
\end{align*}
described as follows.  The $j$th component of the map $\Pi_i^M$ is obtained by forgetting all marked points not in $\Delta_j$.  The map $\Pi_e^M$ is the one induced by the quotient map $\D_m \to \D_m^\Delta$.  We can describe $\Pi_e^M$ in terms of braid diagrams.  Specifically, we collapse---for each $j$---the strands from a single $\Delta_j$ into a single strand; this results in a well-defined braid diagram since (under the above identification) the strands from a single $\Delta_j$ travel alongside each other.  The map $\Pi_i^M$ also has an elementary description in terms of the standard generators of the braid group, where each $\sigma_i$ maps to either the identity or some $\sigma_j$; we leave this description to the reader.

We now give the definitions of $\Pi_i^M$ and $\Pi_e^M$ for an arbitrary un-nested multicurve.  Let $M$ be such a multicurve in $\D_m$.  Let $g \in \B_m$ be a braid so that $g(M)$ is standard.  We define
\begin{align*}
\Pi_i^M &: \Fix_{\B_m}(M) \to \B_{p_1} \times \cdots \times \B_{p_k}, \text{ and} \\
\Pi_e^M &: \Fix_{\B_m}(M) \to \B_{m-p+k},
\end{align*}
via the compositions
\begin{align*}
\Pi_i^{M} &= \Pi_i^{g(M)} \circ \alpha_g , \text{ and} \\
\Pi_e^{M} &=  \Pi_e^{g(M)} \circ \alpha_g.
\end{align*}
where $\alpha_g$ is the inner automorphism of $\B_n$ given by $h \mapsto ghg^{-1}$.  These maps are only defined up to conjugacy.

The map $\Pi_e^M$ is not surjective in general.  For $n$ and $n'$ let $\B_{n,n'}$ denote the subgroup of $\B_{n+n'}$ consisting of braids that fix a given choice of $n'$ marked points (all choices yield isomorphic groups).  The image of $\Pi_e^M$ lies in the subgroup $\B_{m-p,k} \leqslant \B_{m-p+k}$ consisting of the elements that fix the marked points corresponding to the $\Delta_j$.

\begin{lemma}
\label{lem:pkg}
Let $M = \{ c_1, \dots, c_k \}$ be an un-nested multicurve in $\D_m$.  Say that each $c_j$ surrounds exactly $p_j$ marked points, and let $p = p_1 + \cdots + p_k$.  Any choice of homomorphism
\[
\Pi_i^M \times \Pi_e^M : \Fix_{\B_m}(M) \to \left(\B_{p_1} \times \cdots \times \B_{p_k}\right) \times \B_{m-p,k}
\]
is an isomorphism.  
\end{lemma}

\begin{proof}

Since each inner automorphism $\alpha_g$ induces an isomorphism $\Fix_{\B_m}(M) \to \Fix_{\B_m}(g(M))$ we may assume without loss of generality that $M$ is standard.  We would like to define an inverse homomorphism $F$ of $\Pi_i^M \times \Pi_e^M$.  First, we claim that there is a homomorphism $\beta : \B_{m-p,k} \to \Fix_{\B_m}(M)$ that is a section of $\Pi_e^M$ (meaning $\Pi_e^M \circ \beta$ is the identity) and so that $\Pi_i^M \circ \beta$ is the trivial homomorphism.  One way to define $\beta$ is through braid diagrams: for $1 \leq j \leq k$, we replace the $j$th strand with $p_j$ parallel strands.  (A description of this homomorphism $\beta$ in terms of mapping class groups was given by Bell and the third author \cite{bm2}; to obtain $\beta$ from their map $\iota : L_k \to \overline{\textrm{Mod}}(\bar S_k)$ given in their Figure 7 we post-compose $\iota$ with the map $\overline{\textrm{Mod}}(\bar S_k) \to \B_m$ induced by inclusion of $\bar S_k$ into $\D_m$.)  

Let $\iota_j : \B_{p_j} \to \B_m$ be the inclusions induced by the inclusions $\Delta_j \to \D_m$.  Again, these maps can be also described with braid diagrams and also algebraically: the generators $\sigma_1,\dots,\sigma_{p_j-1}$ map to $\sigma_\ell,\dots,\sigma_{\ell+p_j-2}$, respectively, where $\ell$ is the number of the first marked point in $\Delta_j$.  The product $\iota_1\times ... \times \iota_k$ is a section of $\Pi_i^M$, and the composition $\Pi_e^M\circ (\iota_1\times ... \times \iota_k)$ is the trivial homomorphism.  

We define 
\[
F : \left(\B_{p_1} \times \cdots \times \B_{p_k}\right) \times \B_{m-k,p} \to \Fix_{\B_m}(M)
\]
by
\[
F(f_1,\dots,f_k,f_e) = \iota_1(f_1)\iota_2(f_2)\cdots\iota_k(f_k)\beta(f_e).
\]
The map $F$ is a homomorphism because the images of the maps $\iota_1,\dots,\iota_k,\beta$ commute pairwise.  Because $\Pi_e^M \circ \beta$ and $\Pi_i^M \circ (\iota_1\times ... \times \iota_k)$ are the identity and $\Pi_i^M \circ \beta$ and $\Pi_e^M\circ (\iota_1\times ... \times \iota_k)$ are trivial, it follows that $(\Pi_i^M \times \Pi_e^M) \circ F$ is the identity, as desired.
\end{proof}

Given a homomorphism $\rho : \B_n \to \B_m$ and an un-nested multicurve $M$ fixed by $\rho(\B_n)$ as above, we define
\begin{align*}
\rho_i^M &: \B_n \to \B_{p_1} \times \cdots \times \B_{p_k}, \text{ and} \\
\rho_e^M &: \B_n \to \B_{m-p,k}
\end{align*}
by the formulas $\rho_i = \Pi_i^M \circ \rho$ and $\rho_e = \Pi_e^M \circ \rho$.  We refer to these as the \emph{interior} and \emph{exterior} components of $\rho$.  Again, these maps are only well defined up to conjugacy.


\subsection{Braids preserving a multicurve} Let $M = \{ c_1, \dots, c_k \}$ be a multicurve in $\D_m$.  Suppose that each $c_j$ surrounds exactly $p$ marked points, and suppose that $kp=m$.  These conditions imply that every marked point in $\D_m$ lies in the interior of exactly one $c_j$ and in particular that the $c_j$ are un-nested.  Let $\Stab_{\B_m}(M)$ denote the subgroup of $\B_m$ consising of elements that preserve $M$.  We will give a semi-direct product decomposition of  $\Stab_{\B_m}(M)$ that is analogous to the direct product decomposition for $\Fix_{\B_m}(M)$ given in Lemma~\ref{lem:pkg}.  To this end we will define maps $\Pi_e^M$ and $\Pi_i^M$ for a standard multicurve $M$, with the general maps obtained from these in the same way as before.

Consider the homomorphism
\begin{align*}
\Pi_e^M &: \Stab_{\B_m}(M) \to \B_{k},
\end{align*}
given by collapsing to points the disks $\Delta_j$ bounded by the $c_j$.  It follows from the assumptions on the $c_j$ that $\Pi_e^M$ is surjective.  The kernel of $\Pi_e^M$ is the product of the braid groups corresponding to the interiors of the $c_j$ (this follows from \cite[Theorem 3.18]{farbmargalit}).  We thus have a short exact sequence 
\[
1 \to \displaystyle\prod_{i=1}^k B_p \to \Stab_{\B_m}(M) \to \B_k \to 1.
\]
There is a splitting $\B_k \to \Stab_{\B_m}(M)$ given by the map $\beta$ from the proof of Lemma~\ref{lem:pkg}.  The action of $\B_k$ on $\prod B_p$ factors through the standard map $\B_k \to S_k$ and is given by permutation of the factors.  We thus have the following lemma.

\begin{lemma}
\label{lem:pkg2}
Let $M = \{ c_1, \dots, c_k \}$ be an un-nested multicurve in $\D_m$.  Say that each $c_j$ surrounds exactly $p$ marked points, and suppose $m=kp$.  There is a split short exact sequence
\[
1 \to \displaystyle\prod_{i=1}^k B_p \to \Stab_{\B_m}(M) \stackrel{\Pi_e^M}{\to} \B_k \to 1.
\]
In particular we have an isomorphism
\[
\Stab_{\B_m}(M) \cong \B_k \ltimes \displaystyle\prod_{i=1}^k B_p,
\]
where $\B_k$ acts by permuting the factors of $\prod B_p$ according to the standard map $\B_k \to S_k$.  
\end{lemma}

Lemma~\ref{lem:pkg} is a consequence of Lemma~\ref{lem:pkg2}.  It is possible to further generalize Lemma~\ref{lem:pkg2} to the case where $M$ is an arbitrary un-nested multicurve, but we will not need this more general statement.

\section{Transvections and equivalence}
\label{sec:tv}

In Section~\ref{sec:eq} we prove that the relation defined in the introduction is indeed an equivalence relation (Proposition~\ref{prop:eq}) and we also prove our characterization of transvecting elements from the introduction (Proposition~\ref{prop:tv}).  Then in Section~\ref{sec:cent} we give complete descriptions of the centralizers of the images of the standard homomorphisms $\B_n \to \B_{2n}$ (Proposition~\ref{prop:transv}).  We then use this to prove our claim from the introduction that a transvection of a standard homomorphism is a central transvection (Corollary~\ref{cor:central}).  

 
\subsection{Equivalence}
\label{sec:eq}

Let $\rho_0$ and $\rho_1$ be two homomorphisms $\B_n \to G$.  As in the introduction, we define a relation by the rule that $\rho_0 \sim \rho_1$ if and only if there is a transvecting element $t \in G$ for $\rho_0$ and an automorphism $A \in \Aut G$ so that $\rho_1 = A \circ \rho_0^t$.

To prove that this relation is an equivalence relation, we require two facts.  Let $\rho_0$ and $\rho_1$ be  homomorphisms $\rho_0,\rho_1 : \B_n \to G$, let $t$ be a transvecting element for $\rho_1$, let $A \in \Aut G$, and let $u$ be a transvecting element for $A \circ \rho_0$.  We have
\begin{enumerate}
\item $\rho_0 = \rho_1^t \Longleftrightarrow \rho_0^{t^{-1}} = \rho_1$, and
\item $(A \circ \rho_0)^u = A \circ \rho_0^{A^{-1}(u)}$.
\end{enumerate}
Implicit in these two statements are that $t^{-1}$ and $A^{-1}(u)$ are transvecting elements for $\rho_0$.  These are indeed transvecting elements because the hypotheses tell us that $\rho_1^t$ and $(A \circ \rho_0)^u$ are homomorphisms and the maps $\rho_0^{t^{-1}}$ and $A \circ \rho_0^{A^{-1}(u)}$ define the inverse homomorphisms.  Since $A$ is an automorphism, $\rho_0^{A^{-1}(u)}$ is also a homomorphism as desired.  The following proposition is a straightforward consequence of the two facts.

\begin{proposition}
\label{prop:eq}
The relation between homomorphisms $\B_n \to G$ defined by $\sim$ is an equivalence relation.
\end{proposition}

\begin{proposition}
\label{prop:tv}
Let $n \geq 5$ and let $\rho : \B_n \to G$ be a homomorphism.  Then the following are equivalent:
\begin{enumerate}
\item $t \in G$ is a transvecting element for $\rho$,
\item $\rho(\sigma_i) t \rho(\sigma_i)^{-1}$ is independent of $i$, and
\item $t$ centralizes $\rho(\B_n')$.
\end{enumerate}
\end{proposition}

\begin{proof}

We begin with the proof that $(1) \Rightarrow (2)$.  So suppose that $t$ is a transvecting element for $\rho$.  Assume first that $|i-j|>1$. Applying $\rho^t$ to the equation $1=[\sigma_i,\sigma_j]$ we have
\begin{align*}
1 & = [\rho(\sigma_i) t, \rho(\sigma_j) t] \\
& = \rho(\sigma_i) t \rho(\sigma_j) tt^{-1} \rho(\sigma_i)^{-1} t^{-1} \rho(\sigma_j)^{-1} \\
& = \rho(\sigma_i) t \rho(\sigma_j) \rho(\sigma_i)^{-1} t^{-1} \rho(\sigma_j)^{-1} \\
& = \rho(\sigma_i) t \rho(\sigma_i)^{-1} \rho(\sigma_j) t^{-1} \rho(\sigma_j)^{-1},
\end{align*}
as desired.  For the case $j=i+1$, we use the fact that $n \geq 5$.  As such, there is a $k_1$ and $k_2$ so that the sequence $i=k_0,k_1,k_2,k_3=j$ satisfies $|k_{\ell+1} - k_\ell | > 1$ for all $\ell$ and so this case follows from the previous.

For the proof that $(2) \Rightarrow (1)$, we assume that $\rho(\sigma_i) t \rho(\sigma_i)^{-1}$ is independent of $i$, and we need to check that the rule
\[
\sigma_i \mapsto \rho(\sigma_i) t
\]
defines a homomorphism $\B_n \to G$.  The (reverse of the) calculation from the forward direction shows that the images of the $\sigma_i$ satisfy the commutator relations in the standard presentation for the braid group.  It remains to check that they satisfy the braid relations, that is,
\[
\rho(\sigma_i) t \rho(\sigma_{i+1}) t \rho(\sigma_i) t = \rho(\sigma_{i+1}) t \rho(\sigma_i) t \rho(\sigma_{i+1}) t 
\]
for all $i$.  

We claim that $\rho(\sigma_i)^{-1} t \rho(\sigma_i)$ is independent of $i$.  Using the hypothesis that $\rho(\sigma_i) t \rho(\sigma_i)^{-1}$ is independent of $i$ we have
\[
\rho(\sigma_1)t\rho(\sigma_1)^{-1} = \rho(\sigma_3)t\rho(\sigma_3)^{-1}.
\]
Conjugating both sides by $\rho(\sigma_1)^{-1}\rho(\sigma_3)^{-1}$ and using the fact that $\rho(\sigma_1)$ and $\rho(\sigma_3)$ commute gives that $\rho(\sigma_1)^{-1} t \rho(\sigma_1)=\rho(\sigma_3)^{-1} t \rho(\sigma_3)$.  As above, this claim follows since $n \geq 5$.  

Let $X = \rho(\sigma_i)^{-1} t \rho(\sigma_i)$.  This is well defined by the previous claim.  We now claim that $\rho(\sigma_i)^{-1}X\rho(\sigma_i)$ is independent of $i$.  Indeed, we have
\begin{align*}
\rho(\sigma_1)^{-1}X\rho(\sigma_1) &= \rho(\sigma_1)^{-1} \rho(\sigma_3)^{-1} t \rho(\sigma_3) \rho(\sigma_1) \\
&= \rho(\sigma_3)^{-1} \rho(\sigma_1)^{-1} t \rho(\sigma_1) \rho(\sigma_3) \\
&= \rho(\sigma_3)^{-1}X\rho(\sigma_3)
\end{align*}
Again, this claim follows since $n \geq 5$.  

Using the previous two claims and the braid relation we have
\begin{align*}
\rho(\sigma_i) t \rho(\sigma_{i+1}) t \rho(\sigma_i) t &= \rho(\sigma_i) \rho(\sigma_{i+1}) X \rho(\sigma_i) X t \\
&=  \rho(\sigma_i) \rho(\sigma_{i+1}) \rho(\sigma_i)(\rho(\sigma_i)^{-1}X\rho(\sigma_i)) X t \\
&=  \rho(\sigma_{i+1}) \rho(\sigma_i) \rho(\sigma_{i+1})(\rho(\sigma_{i+1})^{-1}X\rho(\sigma_{i+1})) X t \\
&=  \rho(\sigma_{i+1}) \rho(\sigma_i) X\rho(\sigma_{i+1}) X t \\
&=  \rho(\sigma_{i+1}) t \rho(\sigma_i) t \rho(\sigma_{i+1}) t,
\end{align*}
as desired.

We now show that $(2) \Rightarrow (3)$.  As in the proof of $(2) \Rightarrow (1)$, the element $X$ is well defined.  In particular we have
\[
\rho(\sigma_i)^{-1} t \rho(\sigma_i) = \rho(\sigma_1)^{-1} t \rho(\sigma_1)
\]
for all $i$, and so
\[
(\rho(\sigma_1)\rho(\sigma_i^{-1})) t (\rho(\sigma_1)\rho(\sigma_i)^{-1})^{-1} = t
\]
for all $i$.  As in the proof of Lemma~\ref{lem:emptyCRS2}, the $\sigma_1\sigma_i^{-1}$ generate $\B_n'$ for $n \geq 5$, and so the desired implication $(2) \Rightarrow (3)$ follows.  Reversing the calculation for $(2) \Rightarrow (3)$ gives that $(3) \Rightarrow (2)$.  
\end{proof}


\subsection{Centralizers of images of standard maps}
\label{sec:cent}

We now turn to the descriptions of the centralizers of the images of the standard homomorphisms $\B_n \to \B_{2n}$ (Proposition~\ref{prop:transv}). Then we prove Corollary~\ref{cor:central}, which is a refinement of the equivalence relation in the case where one of the homomorphisms is standard.

We require some notation.  Let $\Delta_0$ and $\Delta_1$ be the standard disks in $\D_{2n}$ surrounding the first $n$ and last $n$ marked points, respectively.  The (the isotopy class of) $\Delta_0$ is the support of the image of the standard inclusion map $\B_n \to \B_{2n}$ and $\Delta_0 \cup \Delta_1$ is the support of the image of the diagonal inclusion map.  Let $d_0$ and $d_1$ denote the boundaries of these disks, let $M$ be the multicurve $\{d_0,d_1\}$ and let
\[
\Pi_e^M : \Stab_{\B_{2n}}(M) \to \B_2
\]
be the associated exterior map.  Let $c_1,\dots,c_{2n-1}$ be the standard curves in $\D_{2n}$, so that $H_{c_i} = \sigma_i$ for all $i$.  Let $\tilde \sigma$ denote an element of $\B_{2n}$ with the following properties:
\begin{enumerate}
\item $\tilde \sigma$ interchanges $c_i$ with $c_{n+i}$ for $1 \leq i \leq n$, and
\item $\Pi_e^M(\tilde \sigma)$ is the positive generator for $\PB_2 \cong \Z$. 
\end{enumerate}
The simplest choice of $\tilde \sigma$ is $\tau_n \cdots \tau_1$ where $\tau_i = \sigma_i \cdots \sigma_{i+n-1}$.  We have $\tilde \sigma^2 = z (T_{d_0}T_{d_1})^{-1}$.  Finally, as usual, let $z$ denote the positive generator for the center of $\B_{2n}$.  

Next, let $G_i$ denote the subgroup of $\B_{2n}$ consisting of elements supported outside of $\Delta_0$ (the Dehn twist $T_{d_0}$ lies in this group since it has a representative supported outside a representative of $\Delta_0$).  We define two other subgroup of $\B_{2n}$ as follows:
\begin{align*}
G_d = \langle \tilde \sigma, T_{d_1}, T_{d_2} \rangle \text{ and }
G_f = \langle T_{d_1}, T_{d_2}, z \rangle \cong \Z^3.
\end{align*}
The first group is isomorphic to $\Z \ltimes (\Z \times \Z)$, where the first factor of the semi-direct product acts on the second by interchanging the factors of the direct product.  The group $G_f$ is a subgroup of $G_i$ of index 2.  Finally, we define the subgroup
\[
G_k = \langle z, \sigma_1\sigma_3\cdots\sigma_{2n-1} \rangle \cong \Z^2.
\]

\begin{proposition}
\label{prop:transv}
Let $n \geq 5$ and let $\rho : \B_n \to \B_{2n}$ be a standard homomorphism.  The centralizers in $\B_{2n}$ of $\rho(\B_n)$ and $\rho(\B_n')$ are equal.  For the various choices of $\rho$, the centralizers are equal to the following groups.  
\begin{enumerate}
\item trivial map: $\B_{2n}$
\item inclusion: $G_i$
\item diagonal inclusion: $G_d$
\item flip diagonal inclusion: $G_f$
\item $k$-twist cabling map:  $G_k$
\end{enumerate}
\end{proposition}

\begin{proof}

The case of the trivial map is trivial.  For a given nontrivial standard homomorphism $\rho$, let $G$ be the corresponding group from the statement of the proposition, namely, $G_i$, $G_d$, $G_f$, or $G_k$.  In each case we have $G \subseteq C_{\B_{2n}}(\rho(\B_n))$.  Therefore, to prove the proposition, it suffices to show that $C_{\B_{2n}}(\rho(\B_n')) \subseteq G$, because then we will have a sequence of inclusions
\[
G \subseteq C_{\B_{2n}}(\rho(\B_n)) \subseteq C_{\B_{2n}}(\rho(\B_n')) \subseteq G.
\]
We now treat the four cases in turn.  In each case, $\rho$ is one of the standard maps and $Z = C_{\B_{2n}}(\rho(\B_n'))$.  


\p{Inclusion} Since $Z$ centralizes $\rho(z) = T_{d_0}$, we have that $Z$ is a subgroup of $\Stab_{\B_{2n}}(M)$ where $M = \{d_0\}$.  By the interior/exterior decomposition for $M$ (Lemma~\ref{lem:pkg2}), we have
\[
1 \to \B_n \to \Stab_{\B_{2n}}(M) \to \B_{1,n} \to 1.
\]
If we restrict to $Z$ we obtain
\[
1 \to K \to Z \to B_{1,n} \to 1.
\]
where $K$ is the centralizer of $\B_n'$ in $\B_n$.  The group $\B_{1,n}$ is generated by the image of $G_i$, and the group $K$ is generated by $T_{d_0}$, which is an element of $G_i$.  It follows that $Z=G_i$.  

\p{Diagonal inclusion} Since $Z$ centralizes $\rho(z) = T_{d_0}T_{d_1}$, we have that $Z$ is a subgroup of $\Stab_{\B_{2n}}(M)$ where $M = \{d_0,d_1\}$.  In this case Lemma~\ref{lem:pkg2} gives
\[
1 \to \B_n \times \B_n \to \Stab_{\B_{2n}}(M) \to \B_2 \to 1.
\]
Restricting this short exact sequence to $Z$ we have
\[
1 \to K \to Z \to \B_2 \to 1
\]
where $K$ is the centralizer in $\B_n \times \B_n$ of the image of the diagonal inclusion $\B_n' \to \B_n' \times \B_n$.  The group $\B_2$ in the last sequence is generated by the image of $\tilde \sigma \in Z$.  The group $K$ is $\langle T_{d_0}, T_{d_1} \rangle \cong \Z^2$.  It follows that $Z=G_d$.  

\p{Flip diagonal inclusion} Again, $Z$ lies in $\Stab_{\B_{2n}}(M)$ where $M = \{d_0,d_1\}$.  We claim the image of $Z$ under $\Pi_e^M : \Stab_{\B_{2n}}(M) \to \B_2$ is trivial.  In other words, $Z \leqslant \Fix_{\B_{2n}}(M)$.  Indeed, we have $\rho(\sigma_1) = \sigma_1\sigma_{n+1}^{-1}$.  If we conjugate this image by a braid that interchanges $d_0$ and $d_1$, then we obtain a product of two half-twists, an inverse half-twist in $\Delta_0$ and a half-twist in $\Delta_1$.  Such an element cannot equal $\sigma_1\sigma_{n+1}^{-1}$, and so we have the desired contradiction.  Given the claim, the argument proceeds as in the previous case to show that the centralizer is $G_f$.

\p{Cabling maps} Finally we treat the case of a $k$-twist cabling map.  Let $H$ denote the image of $\B_n'$, and let $Z$ be the centralizer in $\B_{2n}$ of $H$.  We claim that $Z$ preserves the set of standard curves $C = \{c_1,c_3,\dots,c_{2n-1}\}$.  Fix some $c_{2i-1} \in C$.  Since $n \geq 5$ there is a $j$ so that $\sigma_i \sigma_j^{-1}$ is a product of two half-twists.  Consider the element $\sigma_i\sigma_j^{-1} \in \B_n'$.  The canonical reduction system $M$ of $\rho(\sigma_i\sigma_j^{-1})$ consists of six curves: $c_{2i-1}$, $c_{2i+1}$,  $c_{2j-1}$, $c_{2j+1}$ and two curves that surround four marked points each.  Since each element $t$ of $Z$ commutes with $\rho(\sigma_i\sigma_j^{-1})$ it preserves the set of curves in $M$ that surround exactly two marked points, namely $\{c_{2i-1},c_{2i+1},c_{2j-1},c_{2j+1}\}$.  In particular $t(c_{2i-1}) \in C$, whence the claim.

We now have the short exact sequence
\[
1 \to \Z^n \to \Stab_{\B_{2n}}(C) \to \B_n \to 1.
\]
where the $\Z^n$ is generated by $X = \{\sigma_1,\sigma_3,...,\sigma_{2n-1}\}$.  We claim that $\ker \Pi_e^C \cap Z$ is the infinite cyclic group  generated by $\sigma_1\sigma_3\cdots\sigma_{2n-1}$.  Indeed, the group $H$ acts on $\ker \Pi_e^C$ by conjugation, and this action permutes the set $X$.  The image of $H$ in the symmetric group $\Sigma_X$ is the alternating subgroup (because $\B_n'$ surjects onto $A_n$).  The claim follows.  

We claim that the last short exact sequence restricts to $Z$ as follows:
\[
1 \to \Z \to Z \to \Z \to 1
\]
where the kernel is $\Z \cong \langle \sigma_1\sigma_3\cdots\sigma_{2n-1} \rangle$ and the cokernel is $\Z = \langle z \rangle$.  Indeed, we already determined the kernel in the previous claim.  For the cokernel, we use the fact that the image of the restriction $\Pi_e^C|H$ is $\B_n'$, and so $\Pi_e^C(Z)$ must be contained in $C_{\B_n}(\B_n') \cong \Z = \langle z \rangle$.  It follows that $Z=G_k$, as desired.
\end{proof}

The following corollary allows us to deduce the second statement of Theorem~\ref{thm:main} from the first.  

\begin{corollary}
\label{cor:central}
Let $n \geq 5$.  If a homomorphism $\rho : \B_n \to \B_{2n}$ is a transvection of a standard homomorphism $\rho_0$, then $\rho$ is a central transvection of $\rho_0$.  
\end{corollary}

\begin{proof}

Let $\rho : \B_n \to \B_{2n}$.  By Proposition~\ref{prop:tv}, an element $t \in \B_{2n}$ is transvecting for $\rho$ if and only if $t$ centralizes $\rho(\B_n')$.  By Proposition~\ref{prop:transv}, we have that $t$ centralizes $\rho(\B_n)$, as desired.
\end{proof}

Using Corollary~\ref{cor:central}, we can give a geometric description of the transvections of homomorphisms $\B_n \to \B_{2n}$.  More specifically, for $n \geq 5$, we can say that two such homomorphisms differ by a transvection if they only differ on subsurfaces of $\D_{2n}$ where $\B_n$ (through the homomorphism) acts cyclically.  We now explain this in more detail.

Suppose $\rho : \B_n \to \B_{2n}$ is a transvection of the standard homomorphism $\rho_0 : \B_n \to \B_{2n}$, say $\rho = \rho_0^u$.  By Corollary~\ref{cor:central} we have that $u$ is a centrally transvecting element for $\rho_0$.  If $v$ is a transvecting element for $\rho$, it follows that $uv$ is transvecting for $\rho_0$, and hence is centrally trasvecting for $\rho_0$ by Corollary~\ref{cor:central}.  Therefore, while $v$ is not centrally transvecting for $\rho$, it is supported on a subsurface of $\D_{2n}$ on which $\rho$ acts cyclically, namely, by powers of $u$ (this is not quite true for the diagonal inclusion, since there is a centralizing element for the image whose support is all of $\D_{2n}$, but a similar statement is true).


\section{Rotations and curves}
\label{sec:cake}

Recall that $\alpha_1$ and $\alpha_2$ are the periodic elements of $\B_n$ corresponding to rotations of $\D_n$ by $2\pi/n$ and $2\pi/(n-1)$, respectively, and that $\bar \alpha_1$ and $\bar \alpha_2$ are the images in $\bar \B_n$.  In this section we prove two facts about the interactions between the rotations $\bar \alpha_1$ and $\bar \alpha_2$ on one hand, and simple closed curves in $S_{0,n+1}$ on the other hand.  First, in Section~\ref{sec:cake1} we prove Proposition~\ref{prop:cake1}, which states that every curve $c$ in $S_{0,n+1}$ intersects its images under $\bar \alpha_1$ and $\bar \alpha_2$.  Then in Section~\ref{sec:cake2} we prove Proposition~\ref{prop:cake2}, which states that if a curve $c$ is disjoint from each $\bar \alpha_1^i(c)$ with $2 \leq i \leq n-2$ then $c$ surrounds two marked points (and similar for $\bar \alpha_2$).  

\subsection{Curves and primitive rotations}
\label{sec:cake1}

As discussed above, the goal of this section is to prove the following proposition.

\begin{proposition}
\label{prop:cake1}
Let $n \geq 3$, let $k \in \{1,2\}$, let $\epsilon \in \{\pm 1\}$, and let $c$ be an essential simple closed curve in $S_{0,n+1}$.  Then 
\[
i(c,\bar \alpha_k^\epsilon(c)) \neq 0.
\]
\end{proposition}

The basic idea of the proof of Proposition~\ref{prop:cake1} is that if, say, $i(c,\bar \alpha_k(c))=0$, then there exists an arc $\delta$ in $S_{0,n+1}$ that is disjoint from its image $\bar \alpha_k(\delta)$ (take $\delta$ to be any arc in the interior of $c$, that is, the complementary component not containing the distinguished marked point $p$).  We would like to show that this is impossible.  To this end, Lemma~\ref{lemma:cake} below gives a version of Proposition~\ref{prop:cake1} for arcs.  

We first require some setup that will be used throughout this section.  Let $R_1$ be the open unit disk in the Euclidean plane with $n$ marked points equally spaced on a circle centered at the origin, and let $R_2$ be the open unit disk in the Euclidean plane with $n-1$ marked points equally spaced on a circle centered at the origin and with one additional marked point at the origin.  Let $r_1$ denote the clockwise rotation of $R_1$ by $2\pi/n$ and let $r_2$ denote the clockwise rotation of $R_2$ by $2\pi/(n-1)$.  Under appropriate identifications of $R_1$ and $R_2$ with $S_{0,n+1} \! \setminus p$, the maps $r_1$ and $r_2$ correspond to $\bar \alpha_1$ and $\bar \alpha_2$, respectively.

Let $k \in \{1,2\}$.  By an \emph{arc} in $R_k$ we mean the image of an embedding $[0,1] \to R_k$ where the preimage of the set of marked points is $\{0,1\}$.  We say that an arc is \emph{essential} if it is not homotopic (through arcs) into a (small neighborhood of a) marked point.  When we say that two arcs in $R_k$ are disjoint, we mean this in the strictest possible sense: we mean that their interiors are disjoint and also that their endpoints are disjoint.  

A half-open arc in $R_k$ is a proper embedding $[0,1) \to R_k$ where the preimage of the set of marked points is $\{0\}$.  Let $\Delta$ be a collection of $n-1$ disjoint half-open arcs that decompose $R_2$ into $n-1$ fundamental domains for the action of $r_2$; see Figure~\ref{fig:wedge}.  We say that an arc $\delta$ lies in minimal position with $\Delta$ if it has the fewest number of intersections with $\Delta$ in its homotopy class.  By the bigon criterion, this is equivalent to the statement that $\delta$ forms no bigons with $\Delta$.  (It is possible to define a version of $\Delta$ for the action of $r_1$ on $R_1$, but such a $\Delta$ would not be a collection of half-open arcs.)

\begin{lemma}
\label{lemma:cake}
Let $n \geq 3$, and let $\delta$ be an essential arc in $R_2$ that lies in minimal position with $\Delta$.   Then $r_2(\delta)$ and $\delta$ are not disjoint.
\end{lemma}

\begin{proof}

To have a pair of completely disjoint arcs requires $n \geq 4$, so we assume $n \geq 4$.  Assume for the sake of contradiction that $\delta$ is an arc in $R_2$ that lies in minimal position with $\Delta$ and is disjoint from  $r_2(\delta)$.  The collection of arcs $\Delta$ cuts $\delta$ into a sequence of sub-arcs $\delta_1,\dots,\delta_\ell$.  Each $\delta_i$ lies in one fundamental domain (there may be more than one sub-arc in a fundamental domain).  

Because $\delta$ and $\Delta$ do not form any bigons, there are only 6 possibilities for each $\delta_i$ up to homotopy, where homotopies keep the endpoints of $\delta_i$ on $\Delta$ but are allowed to move the endpoints along $\Delta$.  The 6 possibilities are shown in the right-hand side of Figure~\ref{fig:wedge} (we have distorted the fundamental domains there so that they look like rectangles instead of wedges).  As in the figure, we refer to the 6 types of arcs as types $1+$, $1-$, $2+$, $2-$, $3+$, and $3-$.  We further say that an arc is of type 1 if it is of type $1+$ or $1-$, etc.

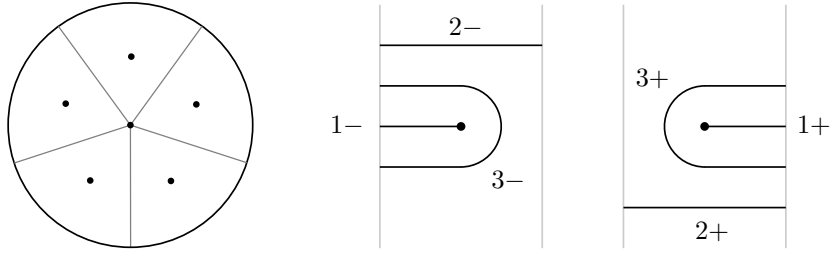
\begin{figure}
\labellist
\small\hair 2pt
\pinlabel $1+$ at 475 72.5
\pinlabel $1-$ at 200 72.5
\pinlabel $2+$ at 415 10
\pinlabel $2-$ at 270 130
\pinlabel $3+$ at 380 100
\pinlabel $3-$ at 295 40
\endlabellist
\includegraphics[scale=.65]{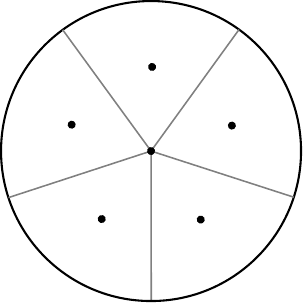} \qquad \qquad
\includegraphics[scale=.85]{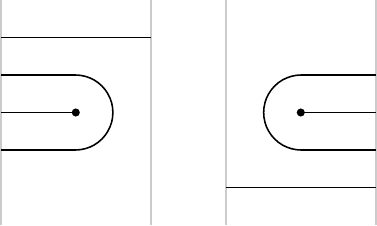}
\caption{Left: the collection of arcs $\Delta$ in the open disk $R_2$; Right: the 6 different types of arcs in a fundamental domain for $r_2$}
\label{fig:wedge}
\end{figure}

Similarly, we may divide the arcs $r_2(\delta)$ into sub-arcs along $\Delta$.  We denote the sub-arcs by $r_2(\delta)_i$.  They of course fall into the same 6 types as the sub-arcs of $\delta$.  Since $\Delta$ is invariant under $r_2$, each $r_2(\delta_i)$ is the sub-arc $r_2(\delta)_i$, and hence each $r_2(\delta)_i$ has the same type as $\delta_i$.

There are exactly two sub-arcs $\delta_i$ of type 1, namely, $\delta_1$ and $\delta_\ell$.  The proof proceeds by analyzing the possibilities for the type of $\delta_i$, beginning with $i=1$ and proceeding inductively.  

We may assume that $\delta_1$ is of type $1+$.  Indeed, if $\delta_1$ is of type $1-$, then we may interchange $\delta$ and $r_2(\delta)$ and interchange the center point of $R_2$ with the exterior puncture in order to obtain the desired situation.  We henceforth assume that $\delta_1$ and $r_2(\delta_1)$ are of type $1+$.  

The arc $\delta_2$ lies in the same fundamental domain as $r_2(\delta)_1$.  Since $\delta \cap r_2(\delta) = \emptyset$ by assumption and since the marked points at the ends of $\delta$ and $r_2(\delta)$ are distinct, it follows that $\delta_2$ must be of type 2.  We assume that it is of type $2+$; the other case is essentially the same (alternatively, we may again interchange the center point of $R_2$ with the exterior puncture).

Again, $r_2(\delta)_2$ is of type $2+$.  Arguing as before, and using only the property that $\delta \cap r_2(\delta) = \emptyset$ we see that $\delta_3$ must also be of type $2+$.  Continuing in this way inductively, we conclude that each $\delta_i$ with $i > 1$ must be of type $2+$.  This contradicts the fact that $\delta_\ell$ is of type 1.   
\end{proof}

\begin{proof}[Proof of Proposition~\ref{prop:cake1}]

We begin with some setup.  For $k \in \{1,2\}$, let $R_k$ be the disk with marked points defined above, and let $R_k^\circ$ be the surface obtained from $R_k$ by removing the marked points; this surface is homeomorphic to a sphere with $n+1$ punctures.

There is a hyperbolic metric on $R_k^\circ$ and a representative $r_k$ of $\bar \alpha_k$ that acts by isometries \cite[Theorem 7.1]{farbmargalit}.  We fix this metric once and for all and refer to it as the hyperbolic metric on $R_k^\circ$ (while the metric depends on $k$, there will be no confusion in what follows).  By a theorem of Brouwer, de K\'er\'ekjart\`o, and Eilenberg \cite{lejb,se,bdk}, any two rotations of a (punctured) sphere through a given angle are conjugate in the group of homeomorphisms, and so we may assume that the $r_k$ given here is the $r_k$ from Lemma~\ref{lemma:cake}.  

Suppose now for the sake of contradiction that $i(c,\alpha_k^\epsilon(c))=0$.  By identifying the complement of the marked points and boundary in $\D_n$ with $R_k^\circ$, it follows that there is a curve $\bar c$ in $R_k^\circ$ so that $i(\bar c,\bar \alpha_k^\epsilon(\bar c))=0$.

Let $\gamma$ be the geodesic representative of $\bar c$ in the hyperbolic metric on $R_k^\circ$.  Since $r_k$ is an isometry of this metric, and since geodesics minimize intersection within homotopy classes, it must be that $\gamma \cap r_k(\gamma) = \emptyset$. It follows that $r_k(\gamma)$ is disjoint from $\gamma$.  For $k=2$ this is impossible by Lemma~\ref{lemma:cake}. 

In the case $k=1$ it follows from the equality $\gamma \cap r_1(\gamma) = \emptyset$ that $\delta$ does not pass through the origin (otherwise $r_1(\delta) \cap \delta \neq \emptyset$ and so interior of $\gamma$ in $R_k^\circ$ intersects its image under $r_1$, implying $\gamma \cap r_1(\gamma) \neq \emptyset$).  Therefore, we may regard $\gamma$ as an arc in $R_2^\circ$ (the one with $n+2$ punctures) with the property that $\gamma \cap r_2(\gamma) = \emptyset$.  Again this is impossible by Lemma~\ref{lemma:cake}.  
\end{proof}


\subsection{Rotations and multicurves}
\label{sec:cake2}

We now proceed to the second and final result of the section.  As above the interior of a curve in $S_{0,n+1}$ is the complementary component not containing the distinguished marked point $p$.  

\begin{proposition}
\label{prop:cake2}
Let $n \geq 5$, and let $M$ be a multicurve in $S_{0,n+1}$.  Suppose that either
\begin{enumerate}
\item $n \geq 5$ and $i(M,\bar \alpha_1^i(M)) = 0$ for $2 \leq i \leq n-2$, or
\item $n \geq 6$ and $i(M,\bar \alpha_2^i(M)) = 0$ for $2 \leq i \leq n-3$.
\end{enumerate}
Then $M$ is a single curve with exactly two marked points in its interior.
\end{proposition}

\begin{proof}

We first prove the proposition in the case of the first hypothesis (about $\alpha_1$), and then prove it under the second hypothesis (about $\alpha_2$).  

Assume that $n \geq 5$ and  $i(M,\bar \alpha_1^i(M)) = 0$ for $2 \leq i \leq n-2$.  We first treat the case where $M$ is (a priori) equal to a single curve $c$.  Let $N$ be the largest even number less than $n$.  Consider the multiset of curves
\[
X = \{c, \bar \alpha_1^2(c), \bar \alpha_1^4(c), \dots, \bar \alpha_1^N(c).\}
\]The elements of $X$ have trivial intersection pairwise.  Indeed, $c$ has trivial intersection with each by assumption, and by applying $\bar \alpha_1^{-2}$ to the other curves simultaneously, we see that $\bar \alpha_1^2(c)$ is disjoint from the curves that come after it, etc.

We claim that the elements of $X$ are also pairwise distinct (in other words the multiset $X$ is a set).  Suppose to the contrary that $\bar \alpha_1^p(c)=\bar \alpha_1^q(c)$ where $p$ and $q$ are even and $0 \leq p < q \leq N$.  Applying $\bar \alpha_1^{-p}$ to both curves we obtain the equality $c = \bar \alpha_1^i(c)$ where $0 < i \leq N$ (here $i=q-p$).  We first treat the case where $i \neq N$.  Applying $\bar \alpha_1$ to both sides of the equality $c = \bar \alpha_1^i(c)$ we obtain the equality $\bar \alpha_1(c) = \bar \alpha_1^{i+1}(c)$.  By the assumptions on $c$ in the statement of the proposition and the assumption $i < N$, we have that $i(c,\bar \alpha_1^{i+1}(c))=0$.  Replacing $\bar \alpha_1^{i+1}(c)$ with $\bar \alpha_1(c)$ we obtain $i(c,\bar \alpha_1(c))=0$.  Since $c$ is essential by assumption, this contradicts Proposition~\ref{prop:cake1}.  The case where $i=N$ is treated in the same way, with $\bar \alpha_1$ replaced by $\bar \alpha_1^{-1}$.  This completes the proof of the claim.

Since each curve of $X$ is the image of $c$ under some mapping class, they all surround the same number of marked points and moreover the corresponding sets of marked points are disjoint.  Since $|X| = N/2+1 > n/3$ for $n \geq 4$ and since the curves are distinct it follows that each curve, in particular the curve $c$, surrounds at most 2 marked points.  Because $c$ is essential, it surrounds exactly 2.  This completes the proof in the special case where $M=c$.

We now prove the general case of the proposition.  Suppose for the sake of contradiction that $M$ has two components $c_1$ and $c_2$ (and possibly others).  We consider the corresponding multisets of curves $X_1$ and $X_2$, defined in the same way as $X$ above.  As above, the elements of $X_i$ are pairwise disjoint and distinct for each $i$.  We conclude as above that each element of each $X_i$ is a curve surrounding exactly two marked points.  

We claim that the elements of the multiset $X_1 \cup X_2$ are pairwise distinct.  Suppose to the contrary that, say, $\bar \alpha_1^u(c_1)=\bar \alpha_1^v(c_2)$; we may assume without loss of generality that $u < v$.  As above we obtain from this the equality $c_1 = \bar \alpha_1^i(c_2)$ with $0 < i \leq N$ and then (since $n \geq 5$) the equality $\bar \alpha_1^{\pm 1}(c_1) = \bar \alpha_1^{i \pm 1}(c_2)$, where $i \pm 1$ is chosen to lie in $[2,n-2]$.  By Proposition~\ref{prop:cake1} we have $i(c_1,\bar \alpha_1^{\pm 1}(c_1)) \neq 0$, and so combining this with the previous equality we have $i(c_1,\bar \alpha_1^{i \pm 1}(c_2)) \neq 0$.  It follows that $i(M,\bar \alpha_1^{i \pm 1}(M)) \neq 0$, contrary to the assumption.  This completes the proof of the claim.  

We complete the proof now in two cases, first for $n$ even and then for $n$ odd.  Assume that $n$ is even.  In this case we have that each element of $X_1$ surrounds exactly two marked points, that the elements of $X_1$ are pairwise distinct and disjoint, and that $|X_1| =n/2$.  In other words, the elements of $X_1$ surround the $n$ marked points of $\D_n$ in pairs.  We also have that the elements of $X_2$ surround two marked points each, and that they are distinct and disjoint from the elements of $X_1$.  This is a contradiction.  The case where $n$ is odd is essentially the same, except that $|X_1|=(n-1)/2$.  This completes the proof assuming the first hypothesis.

Now assume that $n \geq 6$ and $i(M,\bar \alpha_2^i(M)) = 0$ for $2 \leq i \leq n-3$.  Let $q$ be the marked point of $S_{0,n+1}$ that is fixed by $\bar \alpha_2$ and is not equal to $p$.  It must be that $q$ does not lie in the interior of any component of $M$, for otherwise $M$ would not be disjoint from its image under any power of $\bar \alpha_2$.  It must also be that each component of $M$ contains at least two marked points in its exterior for the same reason.  Therefore, we may forget the marked point $q$ and we obtain a multicurve in $S_{0,n}$ satisfying the first hypothesis of the proposition.  Since $q$ was not contained in the interior of any component of $M$, the proposition follows.
\end{proof}


\section{Torsion to torsion}
\label{sec:torsion}

The goal of this section is to prove Proposition~\ref{prop:numthy} below.  By combining the first statement with Lemma~\ref{lem:cycliccyclic}, we obtain Theorem~\ref{thm:lin} from the introduction, which generalizes Lin's result \cite[Corollary 1.16]{linbp}.  We give here a simple, geometric proof.  
 
For the statement, recall that $z$ is the positive generator for $Z(B_n)$ and that $\bar \rho : \B_n \to \bar B_m$ is the homomorphism associated to a given homomorphism $\rho : B_n \to B_m$.

\begin{proposition}
\label{prop:numthy}
Let $n \geq 5$, let $m\geq 1$, and let $\rho : \B_n \to \B_m$ be a homomorphism. Assume that $\bar \rho(z)$ is periodic. 
\begin{enumerate}
\item If $m$ and $m-1$ are both indivisible by $n$ or if both are indivisible by $n-1$ then $\bar \rho$ has cyclic image.   
\item If $m=n$ then, up to replacing $\rho$ by an equivalent homomorphism, we either have that $\bar \rho$ has cyclic image or that $\bar \rho(\alpha_1) = \bar \alpha_1^k$ where $\gcd(k,n)=1$.
\end{enumerate}
\end{proposition}

Note that in the statement of Proposition~\ref{prop:numthy}, the element $\alpha_1$ lies in $\B_n$ and the element $\bar \alpha_1$ lies in $\bar \B_m$.  In what follows, when we refer to an element $\alpha_i$ or $\bar \alpha_i$, we rely on context to specify which group it lies in.  

To prove Proposition \ref{prop:numthy} we require the following lemma, which is a version of the well-suited curve criterion of Lanier and the third author of this paper \cite{laniermargalit}.  In what follows, we denote by $H_c$ the element of $\B_n$ given by a half-twist about a homotopy class of arcs $c$ in $\D_n$.  

 \begin{lemma}
 \label{lem:wscc}
Let $n \geq 5$ and let $f \in \B_n$.  Suppose there is a homotopy class of arcs $c$ in $\D_n$ such that either
\begin{enumerate}
\item $c$ and $f(c)$ have disjoint representatives, or
\item $c$ and $f(c)$ have representatives that share one endpoint and have disjoint interiors.
\end{enumerate}
Then the normal closure of $f$ in $\B_n$ contains $\B_n'$.
\end{lemma}
 
\begin{proof}

The hypotheses imply that $H_cH_{f(c)}^{-1}$ is conjugate in $\B_n$ to either $\sigma_1 \sigma_2^{-1}$ or $\sigma_1 \sigma_3^{-1}$.  Since $H_cH_{f(c)}^{-1}$ is equal to $H_cfH_{c}^{-1}f^{-1}$ it follows that $\sigma_1 \sigma_2^{-1}$ or $\sigma_1 \sigma_3^{-1}$ lies in the normal closure of $f$.  Since $\sigma_1 \sigma_2^{-1}$ and $\sigma_1\sigma_3^{-1}$ are both normal generators for $\B_n'$ (see \cite[Remark 1.10]{linbp}), the lemma follows.
 \end{proof}

\begin{proof}[Proof of Proposition \ref{prop:numthy}]

Suppose that neither $m$ or $m-1$ is a multiple of $n$.  By the classification of periodic elements in $\bar \B_m$, we know that $\bar \rho(\alpha_1)$ is conjugate to a power of either $\bar \alpha_1$ or $\bar \alpha_2$. Since $\bar \alpha_1$ has order $m$ and $\bar \alpha_2$ has order $m-1$, it follows that there is an $\epsilon \in \{m,m-1\}$ so that $\alpha_1^\epsilon$ lies in the kernel of $\bar \rho$. 

Since $\epsilon$ is not a multiple of $n$, it must be that $\alpha_1^\epsilon$ is a non-central power of $\alpha_1$.  All non-central powers of $\alpha_1$ satisfy the hypotheses of Lemma~\ref{lem:wscc}.  Thus, the normal closure of $\alpha_1^\epsilon$, hence the kernel of $\bar \rho$, contains $\B_n'$.  It follows that the image of $\bar \rho$ is cyclic, as desired.

The case where neither $m$ or $m-1$ is a multiple of $n-1$ is essentially the same, with $\alpha_1$ replaced by $\alpha_2$.  

To prove the second statement, assume that $m=n$ and that  $\bar \rho(\alpha_1) = \bar \alpha_1^k$ with $\gcd(k,n)\neq 1$.  In this case there is a $0 < j < n$ such that $\bar \rho(\alpha^j)=1$.  In other words $\alpha^j \in \ker \bar \rho$.  As above, an application of Lemma~\ref{lem:wscc} completes the proof.
\end{proof}

We have the following corollary of the first statement of Proposition~\ref{prop:numthy}.

\begin{corollary}
\label{cor:numthy}
Let $n \geq 5$, let $m \leq 2n$, and assume $m \neq n$.  Let $\rho : \B_n \to \B_m$ be a homomorphism, and assume that $\bar \rho(z)$ is periodic.  Then $\bar \rho$ has cyclic image.
\end{corollary}

\begin{proof}

It is enough to show that $m$ and $n$ satisfy the hypotheses of the first statement of Proposition~\ref{prop:numthy}.  To this end, we assume that one of $\{m,m-1\}$ is divisible by $n$ and show that both of $\{m,m-1\}$ are indivisible by $n-1$.  

Suppose first that $n$ divides $m$.  Since $m \leq 2n$ and $m \neq n$, it must be that $m=2n$.  As $n > 3$, it follows that  $m=2n$ and $m-1=2n-1$ are both indivisible by $n-1$.  

Now suppose $n$ divides $m-1$.  Since $m \leq 2n$ it must be that $m=n+1$.  As $n > 3$, it follows that $m=n+1$ and $m-1=n$ are indivisible by $n-1$, as desired.
\end{proof}


\section{Torsion to pseudo-Anosov}
\label{sec:pA}

The goal of this section is to prove Proposition~\ref{prop:torsiontopA} below.  As above, we denote by $S_{0,n+1}$ a sphere with $n+1$ marked points and we identify the group $\bar \B_n$ with the subgroup of $\Mod(S_{0,n+1})$ fixing one distinguished marked point $p$.  Also, we say that an element of $\bar \B_n$ is pseudo-Anosov if the corresponding element of $\Mod(S_{0,n+1})$ is.  

\begin{proposition}\label{prop:torsiontopA}
Let $m,n \geq 3$ and let $\rho : \B_n \to \B_m$ be a homomorphism.  If $\bar \rho(z)$ is pseudo-Anosov then $\bar \rho$ has cyclic image.
\end{proposition}

Before we can prove the proposition, we will need the following lemma, which concerns the structure of the centralizer of a pseudo-Anosov braid. The first statement has appeared in the literature, for instance in the work of Gonz\'alez-Meneses \cite[Proposition 4.2]{Juan}; for completeness, we give a proof. 

\begin{lemma}
\label{lemma:magic}
If $f \in \bar \B_n$ is pseudo-Anosov, then the centralizer $C_{\bar \B_n}(f)$ is abelian. Further, if $g \in C_{\bar \B_n}(f)$ is pseudo-Anosov, then $C_{\bar \B_n}(f) = C_{\bar \B_n}(g)$.
\end{lemma}

\begin{proof}

In any group, two commuting elements with abelian centralizers have equal centralizers.  Thus, the second statement follows from the first.  It remains to prove the first.

Suppose that $f \in \bar \B_n$ is pseudo-Anosov.  We may regard $f$ as an element of $\Mod(S_{0,n+1})$. Under the action of $\Mod(S_{0,n+1})$ on the space $\PMF(S_{0,n+1})$ of projective measured foliations, the element $f$ has two fixed points, $\F_s$ and $\F_u$, and acts with source-sink dynamics.  Let $\G^*$ denote the subgroup of $\Mod(S_{0,n+1})$ consisting of elements that fix both $\F_s$ and $\F_u$.  Further, let $\G_0^* = \G^* \cap \bar \B_n$.  In other words, $\G_0^*$ is the subgroup of $\G^*$ consisting of elements that fix the distinguished marked point $p$.  

McCarthy in \cite{mccarthy} proved there is a short exact sequence
\[
1 \to F \to \G^* \to \Z \to 1
\]
where $F$ is a finite group.  

Let $C_{\Mod}(f)$ denote the centralizer of $f$ in $\Mod(S_{0,n+1})$.  Because of the source-sink dynamics, we have $C_{\Mod}(f) \subseteq \G^*$.  Since $C_{\bar \B_n}(f) \subseteq C_{\Mod}(f)$, we have $C_{\bar \B_n}(f) \subseteq \G^*$.  Since the elements of $\bar \B_n$ fix the distinguished marked point $p$ we further have $C_{\bar \B_n}(f) \subseteq \G_0^*$.  Thus it suffices to show that the latter is abelian.

To this end, we restrict McCarthy's short exact sequence to $\bar \B_n$:
\[
1 \to F_0 \to \G_0^* \to \Z \to 1.
\]
The group $F_0$---indeed any finite subgroup of $\bar \B_n$---may be regarded as a group of rotations about $p$.  Because distinct elements have different angles of rotation about $p$, no two elements of $F_0$ are conjugate in $\bar \B_n$.  It follows that the conjugation action of  $\G_0^*$ on $F_0$ is trivial.  In other words, $F_0$ is central in $\G_0^*$.  Hence $\G_0^*$ is abelian, as desired. 
\end{proof}

\begin{proof}[Proof of Proposition~\ref{prop:torsiontopA}]

The group $\B_n$ is equal to the centralizer of $z$ and hence $\bar \rho(\B_n)$ maps into the centralizer of the pseudo-Anosov element $\bar \rho(z)$.  By Lemma~\ref{lemma:magic}, the latter is abelian.  Since the abelianization of $\B_n$ is cyclic, the proposition follows.
\end{proof}


\section{Castel's theorem and extensions}
\label{sec:castel}

In this section we give our proof of Castel's classification of homomorphisms $\B_n \to \B_n$.  Castel's original theorem classifies homomorphisms $\B_n \to \B_n$ for $n \geq 6$.  In addition to extending the classification to the case $n \geq 5$ (Theorem~\ref{thm:castel}), we prove as a separate theorem an extension to the case $n=4$, which accounts for the exceptional homomorphism $\B_4 \to \B_3$ (Theorem~\ref{thm:castel4}).  We prove each of these theorems in a separate subsection.

\subsection{Castel's theorem}

The following theorem is our extension of Castel's theorem to the case $n \geq 5$.  

\begin{theorem}
\label{thm:castel}
Let $n \geq 5$.  Any homomorphism $\rho : \B_n \to \B_n$ is equivalent to either the trivial homomorphism or the identity. 
\end{theorem}

Before proceeding to the proof of Theorem~\ref{thm:castel}, we require two preliminary lemmas.

\begin{lemma}
\label{lem:cycliccyclic}
Let $m,n \geq 3$, and let $\rho : \B_n \to \B_m$ be a homomorphism.  Then $\rho$ has cyclic image if and only if $\bar \rho$ has cyclic image.  
\end{lemma}

\begin{proof}

If $\rho$ has cyclic image then $\bar \rho$ must be cyclic, since $\bar \rho$ is the post-composition of $\rho$ with a quotient map.  Now assume that $\bar \rho$ has cyclic image. Then we have an exact sequence
\[
1\to Z(B_m)\cap \rho(\B_n) \to \rho(\B_n)\to \bar \rho(\B_n)\to 1,
\]
so $\rho(\B_n)$ is an extension of a cyclic group by a cyclic group. Since $Z(B_m)\cap \rho(\B_n)$ is central in $\B_m$, we in fact have that $\rho(\B_n)$ is abelian. This implies that $\rho$ factors through the abelianization of $\B_n$, which is $\mathbb{Z}$. Hence $\rho$ has cyclic image. 
\end{proof}

\begin{lemma}\label{lem:emptyCRS}
Let $n \geq 5$ and $m \geq 3$, and let $\rho: \B_n \to \B_m$ a homomorphism.  If the canonical reduction system of 
$\bar \rho(\sigma_1)$ is empty, then $\bar \rho$ has cyclic image.
\end{lemma}

\begin{proof}

The hypothesis implies that $\bar \rho(\sigma_1)$ is either periodic or pseudo-Anosov.  We treat these two cases in turn.

Assume that $\bar \rho(\sigma_1)$ is periodic. Since $\bar \rho(\sigma_3)$ commutes with $\bar \rho(\sigma_1)$, together they generate a finite abelian subgroup of $\bar \B_m$.  As discussed in the proof of Lemma~\ref{lemma:magic}, any finite subgroup of $\bar \B_m$ is conjugate to a subgroup of either $\langle \bar \alpha_1\rangle$ or $\langle \bar \alpha_2\rangle$.  In particular, the group generated by $\bar \rho(\sigma_1)$ and $\bar \rho(\sigma_3)$ is conjugate to a subgroup of either $\langle \bar \alpha_1\rangle$ or $\langle \bar \alpha_2\rangle$.  \color{black}  Since no two distinct elements of $\langle \bar \alpha_1\rangle \cup \langle \bar \alpha_2\rangle$ are conjugate it follows that $\bar \rho(\sigma_1) = \bar \rho(\sigma_3)$.  In other words, $\sigma_1\sigma_3^{-1}\in \ker(\bar \rho)$. Since $n\geq 5$, the commutator subgroup $\B_n'$ is normally generated in $\B_n$ by $\sigma_1\sigma_3^{-1}$,  it follows that $\bar \rho$ has cyclic image.

Now assume that $\bar \rho(\sigma_1)$ is pseudo-Anosov. Since each $\bar \rho(\sigma_i)$ is conjugate to $\bar \rho(\sigma_1)$, the $\bar \rho(\sigma_i)$ are all pseudo-Anosov. Since each $\bar \rho(\sigma_{2i-1})$ commutes with $\bar \rho(\sigma_1)$, they all have the same centralizer by Lemma \ref{lemma:magic}. Since $n \geq 5$, each $\bar \rho(\sigma_{2j})$ commutes with some $\bar \rho(\sigma_{2i-1})$; hence all of the $\bar \rho(\sigma_i)$ have the same centralizer.  In particular, they all commute with each other.  It follows that the image of $\bar \rho$ is abelian.  Since the abelianization of $\B_n$ is cyclic, it follows that $\bar \rho$ has cyclic image. 
\end{proof}

\begin{proof}[Proof of Theorem~\ref{thm:castel}]

We consider three cases, according to whether $\bar \rho(z)$...
\begin{enumerate}
\item is pseudo-Anosov, 
\item is periodic, or
\item has non-empty canonical reduction system.
\end{enumerate}
We treat the three cases in turn.

\bigskip

\noindent \emph{Case 1: $\bar \rho(z)$ is pseudo-Anosov.}  By Proposition~\ref{prop:torsiontopA}, we have that $\bar \rho$ has cyclic image.  By Lemma~\ref{lem:cycliccyclic}, $\rho$ itself has cyclic image. 

\bigskip

\noindent \emph{Case 2:  $\bar \rho(z)$ is periodic.}  Assume $\bar \rho$ is not cyclic.  We will show that $\rho$ is equivalent to the identity.  By Proposition~\ref{prop:numthy}(2), we may assume that $\bar \rho(\alpha_1)$ is equal to $\bar \alpha_1^k$, where $\gcd(k,n)=1$.

 By Lemma~\ref{lem:emptyCRS}, we may assume that the canonical reduction system $M$ of $\bar \rho(\sigma_1)$ is non-empty.  Given this, the key to the proof is to show that $M$ is a single curve surrounding exactly two marked points.  
 
We first claim that $k \in \{\pm 1\}$.  Suppose for contradiction that $k \notin \{\pm 1\}$.  Since $\gcd(k,n)=1$, it follows that $k$ is a unit in $\Z/n$, and so it has a multiplicative inverse $j$.  The inverse $j$ does not lie in $\{\pm 1\}$, and so $\alpha_1^j\sigma_1\alpha_1^{-j}$ commutes with $\sigma_1$.  Applying $\bar \rho$ to the last statement, and using the fact that $\bar \rho(\alpha_1^j) = (\bar \alpha_1^k)^j=\bar \alpha_1$, we conclude that $\bar \alpha_1 \bar \rho(\sigma_1) \bar \alpha_1^{-1}$ commutes with $\bar \rho(\sigma_1)$.  The canonical reduction system of $\bar \alpha_1 \bar \rho(\sigma_1) \bar \alpha_1^{-1}$ is $\bar \alpha_1(M)$.  By the commuting relation, it follows that the multicurves $M$ and $\bar \alpha_1(M)$ have trivial geometric intersection.  This violates Proposition~\ref{prop:cake1}, which says that each curve of $M$ intersects its image $\bar \alpha_1(M)$ under $\bar \alpha_1$.  

By the previous claim, we may assume that $\bar \rho(\alpha_1)=\bar \alpha_1$ (if necessary, we post-compose with the inversion automorphism of $\B_n$). 

Next we claim that $M$ is a single curve surrounding exactly two marked points.   Since $\alpha_1^t \sigma_1 \alpha_1^{-t}$ commutes with $\sigma_1$ for $t \in \{2,\dots,n-2\}$, it follows that $\bar \rho(\alpha_1^t) \bar \rho(\sigma_1) \bar \rho(\alpha_1^{-t})$ commutes with $\bar \rho(\sigma_1)$ for $t \in \{2,\dots,n-2\}$.  By the previous claim, it further follows that $\bar \alpha_1^t \bar \rho(\sigma_1) \bar \alpha_1^{-t}$ commutes with $\bar \rho(\sigma_1)$ for $t \in \{2,\dots,n-2\}$.  This means that $i(\bar \alpha_1^t(M),M)=0$ for $t \in \{2,\dots,n-2\}$.  The claim now follows from Proposition~\ref{prop:cake2}. 

Let $c_i$ denote the canonical reduction system of $\bar \rho(\sigma_i)$ for $1 \leq i \leq n-1$.  By the previous paragraph, $c_1=M$ is a single curve surrounding exactly two marked points.  Since the $\sigma_i$ are all conjugate in $\B_n$, each $c_i$ is a single curve surrounding exactly two marked points.  Since $\sigma_1$ commutes with $\sigma_i$ for $i \geq 3$, it follows that $c_1$ and $c_i$ have trivial intersection for each such $i$.    Also, since $c_3$ is equal to $\bar \alpha_1^2(c_1)$, it must be that $c_1$ and $c_3$ are distinct ($\bar \alpha_1^2$ preserves no set of two marked points).

We claim that there is an $\ell$ so that $\bar \rho(\sigma_i) = H_{c_i}^\ell$ for all $i$.  As the $\sigma_i$ are pairwise conjugate, it suffices to show that $\bar \rho(\sigma_1)$ is a power of $H_{c_1}$.  Since $c_1$ is the canonical reduction system for $\bar \rho(\sigma_1)$, the mapping class $\bar \rho(\sigma_1)$ induces a mapping class $r$ of the component $R$ of $S_{0,n+1} \!\setminus\! c_1$ corresponding to the exterior of $c_1$.  The claim is equivalent to the statement that $r$ is trivial.  By the definition of the canonical reduction system, the mapping class $r$ is either pseudo-Anosov or periodic.  It cannot be pseudo-Anosov because it fixes $c_3$.  Thus $r$ is periodic, and in particular a rotation.  Since $r$ fixes $c_3$, it acts by rotation on the component $R'$ of $R \setminus c_3$ corresponding to the exterior of $c_3$.  As $r$ fixes the three punctures of $R'$ corresponding to $\partial D_n$, $c_1$, and $c_3$, the induced rotation on $R'$ is trivial.  Hence $r$ is trivial, as desired.

To complete the proof, we treat separately two cases, according to whether or not the $c_i$ are pairwise distinct.  Suppose first the $c_i$ are not all distinct.  By the previous claim it follows that $\sigma_i\sigma_j^{-1}$ lies in the kernel of $\bar \rho$.  Since for $n \geq 5$ the normal closure of any $\sigma_i\sigma_j^{-1}$ is $\B_n'$ whenever $i \neq j$, it then follows that $\bar \rho$, hence $\rho$, has cyclic image (Lemma~\ref{lem:cycliccyclic}).  

Now assume that the $c_i$ are pairwise distinct.  Since $\bar \rho(\sigma_1) = H_{c_1}^\ell$, we have that $\rho(\sigma_i) = H_{c_i}^{\ell}z^k$ for all $i$. Since $z$ is central and $\sigma_i$ and $\sigma_{i+1}$ satisfy the braid relation, so too must $H_{c_i}^{\ell}$ and $H_{c_{i+1}}^{\ell}$. Bell and the third author proved \cite[Lemma 4.9]{bellmargalit} that if $H_{c_i}^{\ell}$ and $H_{c_{i+1}}^{\ell}$ satisfy the braid relation and $c_i \neq c_{i+1}$ then $i(c_i,c_{i+1})=2$ and $\ell = \pm 1$.  Up to post-composition by an inner automorphism and the inversion automorphism, we may assume that the $c_i$ are standard and that $\ell=1$, that is, $\rho(\sigma_i) = \sigma_i z^k$.  Finally, we may modify $\rho$ by the central transvection by $z^{-k}$ in order to obtain the identity, completing the proof in the second case.

\bigskip

\noindent \emph{Case 3:  $\bar \rho(z)$ has non-empty canonical reduction system.}  For this case we work directly with $\rho$ instead of $\bar \rho$. Since $\rho(\B_n)$ lies in the centralizer of $\rho(z)$, it follows that, through $\rho$, the group $\B_n$ acts by permutations on the set of components of $M$.  

We claim the action of $\B_n'$ on the set of components of $M$ is trivial.  The number of components of $M$ is at most $n-2$.  Lin proved \cite[Theorem A(c)]{linbp} that any homomorphism $\B_n' \to S_k$ with $k < n$ is trivial if $n \geq 5$; thus the action of $\B_n$ on the set of components of $M$ is cyclic. The claim follows.

Let $c$ be a component of $M$, and say that $c$ surrounds exactly $p$ marked points.  By the previous claim, $\rho(\B_n')$ lies in the group $\Fix_{\B_n}(c)$.  Lin proved that for $n \geq 5$ any homomorphism $\B_n' \to \B_k$ with $k < n$ is trivial  \cite[Theorem A(c)]{linbp}.  It follows that the groups $\Pi_i^c \circ \rho(\B_n')$ and $\Pi_e^c \circ \rho(\B_n')$ are trivial, as they are the images of homomorphisms that satisfy the hypotheses of Lin's theorem.  Since the map $\Pi_i^c \times \Pi_e^c$ is injective (Lemma~\ref{lem:pkg}), it follows that $\rho(\B_n')$ is trivial, and so $\rho$ has cyclic image, as desired.
\end{proof}


\subsection{The case of four strands}

In this section we state and prove the following analogue of Castel's theorem for the case of $\B_4$.  A new proof of this theorem was recently given by Orevkov \cite[Theorem 1.7]{orevkov2}

\begin{theorem}
\label{thm:castel4}
Let $\rho : \B_4 \to \B_4$ be a homomorphism.  Then either
\begin{itemize}
\item $\rho$ factors through the exceptional homomorphism $\B_4 \to \B_3$ or
\item $\rho$ is equivalent to either the trivial map or the identity map.  
\end{itemize}
\end{theorem}

The proof follows along the lines of our proof of Theorem~\ref{thm:castel}.  There are three ways in which the argument differs.  The first issue is that the normal closure of $\sigma_1\sigma_3^{-1}$ in $\B_4$ is not the commutator subgroup, but rather the kernel of the exceptional homomorphism $\B_4 \to \B_3$.  Since we allow for this homomorphism in the statement of Theorem~\ref{thm:castel4}, this means we can use the same arguments as in the proof of Theorem~\ref{thm:castel}, just with a different conclusion.  Specifically, this issue arises in Lemma~\ref{lem:emptyCRS}.

The second issue is that Proposition~\ref{prop:cake2} does not hold as stated for $n=4$.  We will require a specialized version for $n=4$.  There is one additional possibility, namely, that $M$ is a multicurve with two components, each surrounding two marked points.   The proof of this version is essentially the same as for the $n \geq 5$ case, as per Proposition~\ref{prop:cake2}.  

The third issue is that, in the case where we have a cabling $\B_4 \to \B_4$ (meaning that the image preserves a multicurve), we cannot conclude that the image is cyclic, again because of the exceptional homomorphism $\B_4 \to \B_3$.  To deal with this, we first classify homomorphisms $\B_4 \to \B_3$.  We begin with this classification, and then use it to prove Theorem~\ref{thm:castel4}.

Before proceeding to the classifications of homomorphisms $\B_4 \to \B_3$ we first prove the following lemma, which is the analogue of Theorem~\ref{lem:cycliccyclic} for the case $n=4$.

\begin{lemma}
\label{lem:cycliccyclic4}
Let $m \geq 1$ and let $\rho : \B_4 \to \B_m$ be a homomorphism.  Suppose that some $\sigma_i\sigma_j^{-1}$ lies in the kernel of $\bar \rho$.  Then either $\rho$ is cyclic or it factors through the exceptional map $\B_4 \to \B_3$.
\end{lemma}

\begin{proof}

Suppose first that $\sigma_1\sigma_{2}^{-1}$ or $\sigma_2\sigma_{3}^{-1}$ lies in the kernel of $\bar \rho$.  The normal closure of either contains $\B_4'$.  It follows that the image of $\bar \rho$ is cyclic.  As in the proof of Lemma~\ref{lem:cycliccyclic}, the image of $\rho$ is cyclic.

If $\sigma_1\sigma_3^{-1}$ lies in the kernel of $\bar \rho$, then $\rho(\sigma_1\sigma_3^{-1})$ is central in $\B_m$.  Since $\rho(\B_4')$ is contained in $\B_m'$ (the image of a product of commutators is a product of commutators) and since $Z(B_m) \cap \B_m'$ is trivial, it follows that $\sigma_1\sigma_3^{-1}$ lies in the kernel of $\rho$.  Thus, $\rho$ factors through the exceptional homomorphism $\B_4 \to \B_3$.  
\end{proof}

In the proof of the following proposition, we will use the fact that braid groups are Hopfian, which means that every surjective endomorphism is an automorphism.  One way to see this is to use the Magnus embedding $\B_n \to \Aut(F_n)$, the fact that $F_n$ is residually finite, the theorem of Baumslag that the group of automorphisms of a finitely generated residually finite group is residually finite \cite[Theorem 1]{gb}, and the fact that finitely generated residually finite groups are Hopfian \cite[Theorem IV.4.10]{LS}.  

\begin{proposition}
\label{prop:ex}
Let $\rho : \B_4 \to \B_3$ be a homomorphism.  Then either $\rho$ has cyclic image or it factors through the exceptional map $\B_4 \to \B_3$.  Further, if $\rho$ is surjective, then it is equivalent to the exceptional map $\B_4 \to \B_3$. 
\end{proposition}

\begin{proof}

The second statement follows from the first statement and the fact that $\B_3$ is Hopfian.  It remains to prove the first statement.  This proof follows the same outline as the proof of Theorem~\ref{thm:castel}.  We consider three cases for $\bar \rho(z)$, according to whether it is pseudo-Anosov, periodic, or reducible.  

If $\bar \rho(z)$ is pseudo-Anosov then it follows from Proposition~\ref{prop:torsiontopA} and Lemma \ref{lem:cycliccyclic} that $\rho$ has cyclic image.  

If $\bar \rho(z)$, hence $\bar \rho(\alpha_1)$, is periodic, then it follows that $\bar \rho(\alpha_1)$ has order 1, 2, or 3.  We treat these three possibilities in turn.  By Lemma~\ref{lem:cycliccyclic4}, it suffices to show in each case that some $\sigma_i\sigma_j^{-1}$ lies in the kernel of $\bar \rho$.  

If $\bar \rho(\alpha_1)$ has order 1, this means that $\alpha_1$ lies in the kernel of $\bar \rho$, and (as in the proof of Proposition~\ref{prop:numthy}), the kernel of $\bar \rho$ contains $\sigma_1\sigma_2^{-1}$, as desired.  If $\bar \rho(\alpha_1)$ has order 2, $\alpha_1^2$ lies in the kernel of $\bar \rho$, and it follows that the kernel of $\bar \rho$ contains $\sigma_1\sigma_3^{-1}$, as desired.  If $\bar \rho(\alpha_1)$ has order 3, we have that $\alpha_1^3$ lies in the kernel of $\bar \rho$.  Thus there is a conjugate of $\sigma_1\sigma_2^{-1}$ in the kernel of $\bar \rho$, again as desired.

Finally, suppose that $\bar \rho(z)$, hence $\rho(z)$, is reducible.   In this case, there is a multicurve $M$, the canonical reduction system of $\rho(z)$, preserved by $\rho(\B_4)$.  Since $M$ lies in $\D_3$, it must be that $M$ has a single component with exactly two marked points in the interior.  Further, $\rho(\B_4)$ lies in $\Fix_{\B_3}(M)$, which by Lemma~\ref{lem:pkg} is isomorphic to $\B_{1,1} \times \B_2 = \PB_2 \times \B_2 \cong \Z \times \Z$.  In particular, the image of $\rho$ is abelian, hence cyclic.  This completes the proof.  
\end{proof}

We are now ready for the proof of Theorem~\ref{thm:castel4}.  The most difficult part of the proof is the case where $\rho$ is a cabling and the reducing multicurve has two components (this is what comes from the extension of Proposition~\ref{prop:cake2} to the case $n=4$).  This case parallels Case~3 of the proof of Theorem~\ref{thm:main}, which is the most difficult part of that proof.  To deal with this case, we study an element $\varphi \in \B_4$, which is conjugate to---but not equal to---$\alpha_1^2$.  The interplay between these periodic elements leads to a contradiction.

\begin{proof}[Proof of Theorem~\ref{thm:castel4}]

Again, there are three possibilities for $\bar \rho(z)$: it can be pseudo-Anosov, periodic, or reducible.

In the pseudo-Anosov case, the argument is exactly the same as in the proof of Theorem~\ref{thm:castel}.  In the reducible case, the argument is essentially the same, except we must apply Proposition~\ref{prop:ex} in the case where the reducing multicurve has exactly one component, which has exactly three marked points in the interior.

It remains to consider the case where $\rho(z)$, hence $\rho(\alpha_1)$, is periodic.  We may assume that $\rho$ does not have cyclic image and that it does not factor through the exceptional homomorphism $\B_4 \to \B_3$.  By Lemma~\ref{lem:cycliccyclic4}, this is equivalent to the assumption that no $\sigma_i\sigma_j^{-1}$ lies in $\ker \bar \rho$.  

We claim that, up to conjugacy, we have $\bar \rho(\alpha_1)=\bar \alpha_1$.  To prove this, we use the argument from the proof of Proposition~\ref{prop:numthy} in order to rule out all other possibilities.  First, if $\bar \rho(\alpha_1)$ were conjugate to $\bar \alpha_2^j$ for some $j$, then $\alpha_1^3$ would be in the kernel of $\bar \rho$.  The argument of Lemma~\ref{lem:wscc} then shows that $\sigma_1\sigma_2^{-1}$ lies in the kernel of $\bar \rho$, contrary to assumption.  Next, if $\bar \rho(\alpha_1)$ were conjugate to $\bar \alpha_1^2$, then $\sigma_1\sigma_3^{-1}$ would be in the kernel of $\bar \rho$, again contrary to assumption.  The only remaining possibilities are that $\bar \rho(\alpha_1)$ is conjugate to $\bar \alpha_1^{\pm 1}$.  Because the inversion automorphism of $\B_4$ interchanges $\alpha_1$ and $\alpha_1^{-1}$, the claim follows.

We now claim that $\bar \rho(\sigma_1)$ has nonempty canonical reduction system $M$.  By the same argument as in the proof of Lemma~\ref{lem:emptyCRS}, we see that if $\bar \rho(\sigma_1)$ were periodic then $\sigma_1\sigma_3^{-1}$ would lie in $\ker \rho$.  For the case where $\bar \rho(\sigma_1)$ is pseudo-Anosov, we use a variant of the argument used to prove Lemma~\ref{lem:emptyCRS} (this variant also applies in the case $n \geq 5$).  Since $\bar \rho(\sigma_1)$ and $\bar \rho(\sigma_3)$ are commuting, conjugate pseudo-Anosov mapping classes, they have the same invariant foliations and stretch factors.  It follows that $\bar \rho(\sigma_1 \sigma_3^{-1})$, hence $\rho(\sigma_1 \sigma_3^{-1})$ is periodic.  But since $\sigma_1\sigma_3^{-1}$ lies in $\B_4'$, it must be that $\rho(\sigma_1 \sigma_3^{-1})$ lies in $\B_4'$.  The only periodic element of $\B_n'$ is the identity.  This implies that $\rho$ factors through the exceptional map. 

There are three possibilities for $M$:
\begin{enumerate}
\item a single curve surrounding two marked points,
\item a single curve surrounding three marked points, or
\item a pair of curves surrounding two marked points each.
\end{enumerate}

In the first case, we see that $\rho$ is equivalent to the identity map, just as in the proof of Theorem~\ref{thm:castel}.  We treat the other cases in turn.  In the second case we derive a contradiction, and in the third case we show that $\rho$ factors through the standard map $\B_4 \to \B_3$.

Suppose we are in the second case, so $M$ consists a single curve $a_1$ surrounding three marked points. Since $\bar \rho(\alpha_1^2)$ conjugates $\bar \rho(\sigma_1)$ to $\bar \rho(\sigma_3)$, the canonical reduction system of $\bar \rho(\sigma_3)$ is a single curve $a_3$ surrounding three marked points. Since  $\sigma_1$ and $\sigma_3$ commute we have $a_1=a_3=a$.  This implies that $\bar \rho(\alpha_1^2)$ also preserves $a$. This is impossible because $\bar \rho(\alpha_1^2)=\bar \alpha_1^2$ does not preserve any set of three marked points.

We proceed to the third case.  Up to conjugation in $\B_4$, we may assume that the canonical reduction system of $\bar \rho(\sigma_1)$ consists of the standard curves $c_1$ and $c_3$.  Because we are modifying $\rho$, we now only have that $\bar \rho(\alpha_1)$ is conjugate to $\bar \alpha_1$ instead of being equal.

Now, the braid $\bar \rho(\sigma_1)$ acts on $\{c_1,c_3\}$ and so $\bar \rho(\sigma_1^2)$ fixes both curves.  It follows that $\bar \rho(\sigma_1^2) = \sigma_1^a\sigma_3^b$ for some nonzero $a$ and $b$ (here $\sigma_1$ and $\sigma_3$ are considered as elements of $\bar \B_4$).

We claim that $a=b$.  To this end, let $\varphi$ be a conjugate of $\alpha_1^2$ that commutes with $\sigma_1$ (this exists because $\alpha_1^2$ itself commutes with $\sigma_2$, and the latter is conjugate to $\sigma_1$).  Since $\varphi$ commutes with $\sigma_1$, its image $\bar \rho(\varphi)$ acts on $\{c_1,c_3\}$.  Since $\bar \rho(\varphi)$ is conjugate to $\bar \alpha_1^2$, it cannot preserve two disjoint curves, which implies that $\bar \rho(\varphi)$ must interchange $c_1$ and $c_3$.  We now see that 
\[
\bar \rho(\sigma_1^2)=\bar \rho(\varphi \sigma_1^2 \varphi^{-1})=\sigma_1^b\sigma_3^a.
\]
Since $\sigma_1$ and $\sigma_3$ generate a free abelian group, the claim follows.

Now, $\bar \rho(\sigma_1^2)$ is conjugate to $\bar \rho(\sigma_3^2)$ and commutes with it.  But the only conjugate of $(\sigma_1\sigma_3)^a$ that commutes with it is itself (here we use that fact that commuting elements have disjoint canonical reduction systems).  It follows that $(\sigma_1\sigma_3^{-1})^2$ is in the kernel of $\bar \rho$.  Thus, $\rho(\sigma_1\sigma_3^{-1})$ is a torsion element in $\B_4$.  The latter is torsion free, and so in fact $\sigma_1\sigma_3^{-1}$ lies in the kernel of $\rho$.  This means that $\rho$ factors through the standard map $\B_4 \to \B_3$.  The theorem follows.
\end{proof}


\section{The group of even braids}
\label{sec:castel2}

Recall that $\B_n^2$ is the group of even braids.  The goal of this section is to prove Theorem~\ref{thm:castel2}, our extension of Theorem~\ref{thm:castel}.  For the statement, a central transvection of a map $\rho: \B_n^2 \to \B_n$ is a map $\rho^t : \B_n^2 \to \B_n$ given by $\rho^t(g) = \rho(g)t^{L^2(g)}$ where $t$ lies in the centralizer of $\rho(\B_n^2)$ and $L^2 : \B_n^2 \to \Z$ is the abelianization (cf. Lemma~\ref{lem:bn2abel}).  

\begin{theorem}
\label{thm:castel2}
Let $n\geq 5$, and let $\rho: \B_n^2\to \B_n$ be a homomorphism. Then $\rho$ is almost-conjugate to a central transvection of either the trivial map or the standard inclusion. 
\end{theorem}

As for the case of the full braid group we denote by $\bar \rho$ the homomorphism $\bar \rho : \B_n^2 \to \bar \B_n$ associated to a homomorphism $\rho : \B_n^2\to \B_n$.  Before proceeding to the proof of Theorem~\ref{thm:castel2} we require a series of lemmas.

\begin{lemma}
\label{lem:bn2abel}
Let $n\geq 2$. The commutator subgroup of $\B_n^2$ is $\B_n'$ and the abelianization $\B_n^2/\B_n'$ is isomorphic to $\Z$.
\end{lemma}

\begin{proof}

Since $\B_n'$ is contained in $\B_n^2$, and since $\B_n^2/\B_n'\subset \B_n/\B_n'\cong \Z$ is equal to the subgroup $2\Z$, we have a short exact sequence
\[
1\to \B_n'\to \B_n^2\to 2\Z\to 1.
\]
By the right-exactness of the abelianization functor, we obtain a further short exact sequence
\[
(B_n')^{ab}\to (\B_n^2)^{ab}\to 2\Z\to 1.
\]
Since $\B_n'$ is perfect, we have $(B_n')^{ab}=1$, and so the map $(\B_n^2)^{ab}\to 2\Z$ is an isomorphism. This shows that the abelianization of $\B_n^2$ is cyclic, and that the kernel of the abelianization map $\B_n^2\to 2\Z$ is equal to $\B_n'$. 
\end{proof}

\begin{lemma}
\label{lem:conjbn'}
Let $n\geq 3$, and let $f$ be an element of $\B_n$ that commutes with some half-twist $h$.  Then the set of $\B_n$-conjugates of $f$ is equal to the set of $\B_n'$-conjugates of $f$.  Both are equal to the set of $\B_n^2$-conjugates of $f$.
\end{lemma}

\begin{proof}

Let $g \in \B_n$ and consider the conjugate $gfg^{-1}$ of $f$. Let $\ell$ denote the signed word length of $g$.  Then $gh^{-\ell}$ lies in $\B_n'$ and $(gh^{-\ell})f(gh^{-\ell})^{-1}$ is equal to $gfg^{-1}$.  Since $\B_n' \subseteq \B_n^2$, the second statement follows.
\end{proof}

\begin{lemma}
\label{lem:bn2nc}
Let $n\geq 5$. The group $\B_n'$ is the normal closure in $\B_n'$ of $\sigma_1\sigma_3^{-1}$.  Similarly, $\B_n'$ is the normal closure in $\B_n'$ of $\sigma_1\sigma_2^{-1}$.
\end{lemma}

\begin{proof}

We already know that $\B_n'$ is generated by the $\B_n$-conjugates of $\sigma_1\sigma_3^{-1}$, and also that $\B_n'$ is generated by the $\B_n$-conjugates of $\sigma_1\sigma_2^{-1}$. The lemma thus follows from Lemma~\ref{lem:conjbn'}.  
\end{proof}

The periodic braids $\alpha_1$ and $\alpha_2$ have signed word length equal to $n-1$ and $n$, respectively.   We conclude that $\alpha_1\in \B_n^2$ if and only if $n$ is odd and that $\alpha_2\in \B_n^2$ if and only if $n$ is even. Also, we see that $z = \alpha_1^n$ lies in $\B_n^2$ for all $n$.   

The following proposition is a version of Proposition~\ref{prop:numthy}(2) for $B_n^2$.  

\begin{proposition}
\label{lem:numthy2}
Let $n \geq 5$.  Let $\rho : \B_n^2 \to \B_n$ be a homomorphism.  Assume that $\rho$ does not have cyclic image and that $\rho(z)$ is periodic.  Up to replacing $\rho$ by a conjugate homomorphism, the following statements hold.
\begin{enumerate}
\item If $n$ is odd, then $\bar \rho(\alpha_1)=\bar \alpha_1^k$ with  $\gcd(k,n)=1$. 
\item If $n$ is even, then $\bar \rho(\alpha_2)=\bar \alpha_2^k$ with  $\gcd(k,n-1)=1$.
\end{enumerate}
\end{proposition}

\begin{proof}

We prove the second statement only; the proof of the first statement is similar.  As in the proof of Proposition~\ref{prop:numthy}, we may use the classification of periodic elements in $\bar \B_n$ in order to assume without loss of generality that either $\bar \rho(\alpha_2)=\bar \alpha_1^k$ or $\bar \rho(\alpha_2)=\bar \alpha_2^k$.  Also, as in the same proof, it suffices to show that if $\bar \rho(\alpha_2)$ is equal to $\bar \alpha_1^k$ or if $\bar \rho(\alpha_2)$ is equal to $\bar \alpha_2^k$ with $\gcd(k,n-1)\neq 1$ then $\bar \rho$ has cyclic image.

Suppose first that $\bar \rho(\alpha_2) = \bar \alpha_1^k$.  Because $\bar \alpha_1$ has order $n$ we may assume that $0 \leq k < n$. Since $\alpha_1^n = \alpha_2^{n-1}$ we have that 
\[
\bar \rho (\alpha_1^2)^{n/2} = \bar \rho(\alpha_2)^{n-1}.
\]
Combining this with the equality $\bar \rho(\alpha_2)=\bar \alpha_1^k$, we conclude that
\[
\bar \rho (\alpha_1^2)^n = \bar \alpha_1^{2k(n-1)}.
\]
As in the proof of Proposition~\ref{prop:numthy}(2), we consider the composition
\[
\bar \B_n \stackrel{p}{\to} \Z/n(n-1) \to \Z/n.
\]
We apply this composition to the last equality.  The left-hand side of the equality maps to 0 and the right-hand side maps to $2k(n-1)^2 \equiv 2k$. Thus $k$ lies in $\{0,n/2\}$.  In either case $\alpha_2^2$ lies in the kernel of $\bar\rho$, and so by Lemma~\ref{lem:wscc}, we have that $\bar \rho$ has cyclic image, as desired.

We next show that if $\bar \rho(\alpha_2)=\bar \alpha_2^k$ with $\gcd(k,n-1)\neq 1$ then $\bar \rho$ has cyclic image.  Since $\bar \alpha_2$ has order $n-1$ we may assume that $0 \leq k < n-1$.  Let $t = (n-1)/\gcd(n-1,k)$; by the assumption that $\gcd(n-1,k) \neq 1$ we have that $0 < t < n-1$.  We also have $(n-1)|kt$.   We then have that $\bar \rho(\alpha_2^t)$ is equal to $(\bar \alpha_2^k)^t$.  Since $(n-1)|kt$, the latter is the identity, which means that $\alpha_2^t\in \ker \bar \rho$.  By Lemma~\ref{lem:wscc}, the map $\bar \rho$ has cyclic image, as desired.
\end{proof}

\begin{lemma}
\label{lem:emptyCRS2}
Assume that $n\geq 5$. Let $\rho: \B_n^2\to \B_n$ be a homomorphism. If $\bar \rho(\sigma_1^2)$ has trivial canonical reduction system then $\bar \rho$ has cyclic image. 
\end{lemma}

\begin{proof}

Suppose $\bar \rho(\sigma_1^2)$ has trivial canonical reduction system.  Then $\bar \rho(\sigma_1^2)$ is either periodic or pseudo-Anosov. We consider the two cases separately.   The proof follows the same outline as the proof of Lemma~\ref{lem:emptyCRS}.

Assume that $\bar \rho(\sigma_1^2)$ is pseudo-Anosov.  By Lemma~\ref{lem:conjbn'}, the braids $\sigma_{i}^2$ are pairwise conjugate in $\B_n^2$.  Because of this, the $\bar \rho(\sigma_{i}^2)$ are all pseudo-Anosov.  The $\bar \rho(\sigma_{2i-1}^2)$ also commute pairwise, and so Lemma~\ref{lemma:magic} implies that they all have equal centralizers.  Since $n \geq 5$, each $\bar \rho(\sigma_{2i}^2)$ commutes with some $\bar \rho(\sigma_{2i-1}^2)$, and so it further follows that all of the $\bar \rho(\sigma_i^2)$ have the same centralizer in $\bar \B_n$.  In particular, they all commute.  

Each $\bar \rho(\sigma_1\sigma_k^{-1})$ commutes with either $\bar \rho(\sigma_1^2)$ (when $k > 2$) or $\bar \rho(\sigma_4^2)$ (when $k=2$).  It follows that each $\bar \rho(\sigma_1\sigma_k^{-1})$ lies in the centralizer of $\bar \rho(\sigma_1^2)$ in $\bar \B_n$.  

Gorin--Lin~\cite[p. 7]{linbp} gave a finite presentation for $\B_n'$.  A consequence of their presentation is that $\B_n'$ is generated by the braids $\{\sigma_i\sigma_1^{-1} \mid 2 \leq i \leq n-1\}$ (their generating set includes elements $v$ and $w$, but their relations (1.16) and (1.20) show that these elements are products of the other generators).  It follows that the group $\B_n^2$ is generated by these and the $\sigma_i^2$.  

We already showed that the image of each of the generators for $\B_n^2$ given in the previous paragraph lies in the centralizer of $\bar \rho(\sigma_1^2)$ in $\bar \B_n$.  By Lemma~\ref{lemma:magic}, this centralizer is abelian.  Therefore, the image of $\bar \rho$ is abelian.  It follows that the image of $\rho$ is abelian, hence it has cyclic image by Lemma~\ref{lem:bn2abel} and the analogue of Lemma~\ref{lem:cycliccyclic} for $\B_n^2$ (the proof is the same).  

Next, we assume that $\bar \rho(\sigma_1^2)$ is periodic. As in the proof of Lemma~\ref{lem:emptyCRS}, we have that $\bar \rho(\sigma_1^2)$ and $\bar \rho(\sigma_3^2)$ generate a finite abelian subgroup of $\bar \B_n$ that is conjugate to a subgroup of either $\langle \bar \alpha_1\rangle$ or $\langle \bar \alpha_2\rangle$. Since $\sigma_1^2$ is conjugate to $\sigma_3^2$ in $\B_n^2$ by Lemma~\ref{lem:conjbn'}, we must have $\bar \rho(\sigma_1^2) = \bar \rho(\sigma_3^2)$. In other words, $\bar \rho(\sigma_1\sigma_3^{-1})^2 = 1$. This implies that $\rho(\sigma_1\sigma_3^{-1})^2 = z^m$ for some $m$, hence that $z^m$ lies in the commutator subgroup of $\B_n$. It follows that $m=0$.  In other words, $\rho(\sigma_1\sigma_3^{-1})^2 = 1$. Since $\B_n$ is torsion-free, we have that $\rho(\sigma_1\sigma_3^{-1}) =1$. We may then conclude that all $\bar \rho(\sigma_i^2)$ are equal and thus that $\rho$, hence $\bar \rho$, has cyclic image.
\end{proof}

The following lemma (and a proof) appears in the paper by the second and third authors \cite[Lemma 3.2]{bnprime}.

\begin{lemma}
\label{lem:roots}
Let $n \geq 4$, let $c$ and $d$ be disjoint curves in $\D_n$ that surround exactly two marked points each, and suppose that the braid $(H_cH_d^{-1})^\ell$ has a $k$th root $f$.  Then $\ell$ is divisible by $k$ and
\[
f = (H_cH_d^{-1})^{\ell/k}.
\]
\end{lemma}

We require one final lemma.

\begin{lemma}
\label{lem:bnpgens}
Let $n \geq 5$, and let $\rho : \B_n^2 \to \B_n$ be a homomorphism.  Suppose that $\bar \rho(\sigma_i^2) = \sigma_i^2$ for all $i$.  Then $\rho$ is a central transvection of the standard inclusion.
\end{lemma}

\begin{proof}

The group $\B_n^2$ is generated by $\B_n'$ together with the $\sigma_i^2$.  By assumption, and the  fact that the $\sigma_i^2$ are conjugate in $\B_n^2$, there is an $s$ so that $\rho(\sigma_i^2) = \sigma_i^2 z^s$ for all $i$. Since the signed word length of each element of $\B_n'$ is 0, it suffices to show that the restriction $\rho|\B_n'$ is the standard inclusion.  

As in the proof of Lemma~\ref{lem:emptyCRS2},  $\B_n'$ is generated by the braids $\{\sigma_i\sigma_1^{-1} \mid 2 \leq i \leq n-1\}$.  To complete the proof we will show that $\rho(\sigma_i\sigma_j^{-1}) = \sigma_i\sigma_j^{-1}$ for all $1 \leq i < j \leq n-1$.  

If $|i-j| > 1$ then 
\[
\rho(\sigma_i\sigma_j^{-1})^2 = \rho(\sigma_i^2\sigma_j^{-2}) = \sigma_i^2\sigma_j^{-2} = (\sigma_i\sigma_j^{-1})^2.
\]
It now follows from Lemma~\ref{lem:roots} that $\rho(\sigma_i\sigma_j^{-1})=\sigma_i\sigma_j^{-1}$, as desired.

If $|i-j|=1$, we choose some $k$ with $|i-k| > 1$ and $|j-k|>1$.  We then have
\[
\rho(\sigma_i\sigma_j^{-1}) = \rho(\sigma_i\sigma_k^{-1})\rho(\sigma_k\sigma_j^{-1}) = (\sigma_i\sigma_k^{-1})(\sigma_k\sigma_j^{-1}) = \sigma_i\sigma_j^{-1},
\]
as desired.
\end{proof}

\begin{proof}[Proof of Theorem~\ref{thm:castel2}]

As in the proof of Theorem~\ref{thm:castel} we consider three cases, according to whether $\bar \rho(z)$...
\begin{enumerate}
\item is pseudo-Anosov,
\item is periodic, or
\item has non-empty canonical reduction system.
\end{enumerate}

\bigskip

\noindent \emph{Case 1: $\bar \rho(z)$ is pseudo-Anosov.}  By the same argument as in Proposition~\ref{prop:torsiontopA}, we have that $\bar \rho$ has cyclic image (the image of $\bar \rho$ lies in lies in $C_{\bar \B_n}(\bar \rho(z))$, which is abelian by Lemma~\ref{lemma:magic}).  Lemma~\ref{lem:bn2abel} gives that $(B_n^2)' = B_n'$ and $\B_n^2/(B_n^2)' \cong \Z$.  Thus, the image of $\bar \rho$ is cyclic.  By the same argument as in Lemma~\ref{lem:cycliccyclic}, $\rho$ itself has cyclic image. 

\bigskip

\noindent \emph{Case 2: $\bar \rho(z)$ is periodic.}  Assume that $\rho$ does not have cyclic image.  As in Case 1, this implies that the image of $\bar \rho$ is not cyclic.  

We first claim that, up to replacing $\rho$ with a conjugate homomorphism, $\bar \rho(\sigma_i^2) = H_{c_i}^\ell$, where each $c_i$ is a curve surrounding exactly two marked points.  The argument is essentially the same as the argument in the first seven paragraphs of Case 2 of the proof of Theorem~\ref{thm:castel}, except with $\bar \rho(\sigma_1)$ replaced by $\bar \rho(\sigma_1^2)$, with Proposition~\ref{prop:numthy}(2) replaced by Lemma~\ref{lem:numthy2}, with Lemma~\ref{lem:emptyCRS} replaced by Lemma~\ref{lem:emptyCRS2}, and---in the case of $n$ even---with $\alpha_1$ replaced by $\alpha_2$.  As in that proof, the steps are:
\begin{enumerate}
\item $\bar \rho(\alpha_1) = \bar \alpha_1^k$ with $\gcd(k,n)=1$,
\item the canonical reduction system $M$ of $\bar \rho(\sigma_1^2)$ is nonempty,
\item $k = 1$,
\item $M$ consists of a single curve $c_1$ surrounding two marked points, and
\item $\bar \rho(\sigma_i^2) = H_{c_i}^\ell$.
\end{enumerate}
In order, these steps rely on Lemma~\ref{lem:numthy2}, Lemma~\ref{lem:emptyCRS2}, Proposition~\ref{prop:cake1}, Proposition~\ref{prop:cake2}, and the Nielsen--Thurston classification theory.

By Lemma~\ref{lem:conjbn'}, we have for each $1 \leq i \leq n-1$ that $\sigma_i^2$ is conjugate to $\sigma_1^2$ in $\B_n^2$.  Combining this with the previous claim, it follows that there exist curves $c_i$, each surrounding exactly two marked points, so that $\bar \rho(\sigma_i^2)=H_{c_i}^\ell$ for $1 \leq i \leq n-1$. 

As in the proof of Case 2 of Theorem~\ref{thm:castel}, we continue by treating two cases, according to whether or not the $c_i$ are pairwise distinct.  

\medskip

Suppose first the $c_i$ are not distinct, that is, $c_j=c_k$ for some $j < k$.  We claim the $c_i$ are all equal.  Indeed, if $k=j+1$, then by the same argument used to prove Lemma~\ref{lem:conjbn'}, the ordered pair $(\sigma_j^2,\sigma_k^2)$ is conjugate in $\B_n^2$ to each pair $(\sigma_i^2,\sigma_{i+1}^2)$, and the claim follows.  Similarly, if $k>j+1$ then $(\sigma_j^2,\sigma_k^2)$ is conjugate in $\B_n^2$ to each pair $(\sigma_i^2,\sigma_{i+\ell}^2)$ with $\ell \geq 2$.  Since $n \geq 5$, the claim again follows.  Indeed, given distinct $p,q \in \{1,\dots,n-1\}$, there is a sequence $p=p_0,\dots,p_N=q$ with $|p_i-p_{i+1}| > 1$.  Thus, the $c_{p_i}$ are all equal, as desired.

Given the claim we obtain that $\bar \rho(\sigma_1^2\sigma_3^{-2})=1$.  It follows that $\rho(\sigma_1^2\sigma_3^{-2})=1$.  Since $\B_n$ is torsion free, $\rho(\sigma_1\sigma_3^{-1})=1$.  We also obtain $\rho(\sigma_1\sigma_2^{-1})=\rho(\sigma_1\sigma_4^{-1})\rho(\sigma_4\sigma_2^{-1})=1$. Thus the image of $\rho$ is cyclic, contrary to assumption.

\medskip

We may now assume henceforth that the $c_i$ are pairwise distinct. We claim that $\ell$ is even and that
\[
\rho(\sigma_2\sigma_4^{-1})=(H_{c_2}H_{c_4}^{-1})^{r}
\]
where $\ell=2r$, and  that $c_2$ is disjoint from $c_4$.  As above, we have $\bar \rho(\sigma_2^2) = H_{c_2}^\ell$ and so $\rho(\sigma_2^2) = H_{c_2}^\ell z^s$.  Since $\sigma_2^2$ commutes with $\sigma_4^2$, we have that $c_2$ is disjoint from $c_4$, which is the second statement of the claim. It follows that $H_{c_2}$ and $H_{c_4}$ commute, and so
\[
\rho(\sigma_2\sigma_4^{-1})^2 = \rho(\sigma_2^2\sigma_4^{-2}) = \rho(\sigma_2^2)\rho(\sigma_4^{-2}) =  (H_{c_2}H_{c_4}^{-1})^{\ell}.
\]
It follows that $\rho(\sigma_2\sigma_4^{-1})$ is equal to a square root of $(H_{c_2}H_{c_4}^{-1})^{\ell}$.  The claim follows now from Lemma~\ref{lem:roots}.

Our next claim is that, up to conjugation of $\rho$, we have $\ell=2r=2$ and $i(c_1,c_2)=2$.  Since $\sigma_4$ commutes with $\sigma_1$ and $\sigma_2$, and since $\sigma_1$ and $\sigma_2$ satisfy the braid relation, the braids $\sigma_1\sigma_4^{-1}$ and $\sigma_2\sigma_4^{-1}$ also do.  It follows that their $\rho$-images satisfy the braid relation.  Since $c_4$ is disjoint from $c_1$ and $c_2$ (the previous claim), it follows that $H_{c_1}^r$ and $H_{c_2}^r$ satisfy the braid relation.  As in Case 2 of the proof of Theorem~\ref{thm:castel}, the claim then follows from the result of Bell and the third author cited there.

We now claim that, up to post-composing $\rho$ by an automorphism of $\B_n$, we have $\rho(\sigma_i^2) = \sigma_i^2 z^s$. By the argument used to prove Lemma~\ref{lem:conjbn'}, each ordered pair $(\sigma_i^2,\sigma_{i+1}^2)$ is conjugate in $\B_n^2$ to $(\sigma_1^2,\sigma_2^2)$.  Since $i(c_1,c_2)=2$, it follows that $i(c_i,c_{i+1})=2$ for all $1 \leq i \leq n-1$.  For $1 \leq i ,j \leq n-1$ the braids $\sigma_{2i}^2$ and $\sigma_{2j}^2$ commute, and so $i(c_{2i},c_{2j})=0$.  Combining the last two sentences, we conclude the claim as in the proof of Case 2 of Theorem~\ref{thm:castel}.   

Using the previous claim, an application of Lemma~\ref{lem:bnpgens} completes the proof of Case 2.  

\bigskip

\noindent \emph{Case 3: $\bar \rho(z)$ has nonempty canonical reduction system.}  
Let $M$ denote the canonical reduction system of $\bar \rho(z)$.  As in Case 3 of the proof of Theorem~\ref{thm:castel}, we have that $\rho(\B_n')$ lies in $\Fix_{\B_n}(M)$.  Also as in that proof, $\Pi_i^c \circ \rho(\B_n')$ and $\Pi_e^c \circ \rho(\B_n')$ are trivial by the work of Lin.  This implies that $\rho$ has cyclic image.  This completes the proof.
\end{proof}


\section{Cablings}
\label{sec:cabling}

As in Section~\ref{sec:overview}, a map $\rho : \B_n \to \B_{2n}$ is a \emph{2-fold cabling map} if there is a multicurve $M$ in $\D_{2n}$ that has $n$ components and so that the action of $\rho(B_n)$ on the set of components of $M$ is (conjugate to) the standard action through $B_n\to S_n$ the symmetric group.  We refer to $M$ as the \emph{cabling multicurve} for $\rho$.  We note that each component of $M$ surrounds exactly two marked points in $\D_{2n}$.  The goal of this section is to prove the following proposition.

\begin{proposition}
\label{prop:cabling}
Let $n \geq 2$ and let $\rho : \B_n \to \B_{2n}$ be a 2-fold cabling map.  Then $\rho$ is equivalent to one of the standard $k$-twist cabling maps.
\end{proposition}

Let $C = \{c_1,c_3,\dots,c_{2n-1}\}$ be the standard multicurve in $\D_{2n}$, so that $H_{c_i} = \sigma_i$.  By Lemma~\ref{lem:pkg2} we have associated to $C$ a homomorphism $\Pi_e^C : \Stab_{\B_{2n}}(C) \to \B_n$.  

\begin{lemma}
\label{lem:cable}
Let $n \geq 2$ and let $\rho : \B_n \to \B_{2n}$ be a 2-fold cabling map.  Up to equivalence, $\rho$ satisfies the following:
\begin{enumerate}
\item $\rho$ is a 2-fold cabling map with cabling multicurve $C= \{c_1,c_3,\dots,c_{2n-1}\}$, and
\item $\rho$ is a section of the map $\Pi_e^C : \Stab_{\B_{2n}}(C) \to \B_n$.
\end{enumerate}
\end{lemma}

\begin{proof}

Up to conjugation of $\rho$, we may assume that the cabling multicurve for $\rho$ is $C$, as per the first statement.  We will modify $\rho$ so that it also satisfies the second.  By Lemma~\ref{lem:pkg2}, we have a split short exact sequence
\[
1 \to \prod_{i=1}^n \B_2 \to \Stab_{\B_{2n}}(C) \stackrel{\Pi_e^C}{\to} \B_n \to 1.
\]
Since $C$ is the cabling multicurve for $\rho$ we have $\rho(B_n)\subset \Stab_{\B_{2n}}(C)$. The composition $\Pi_e^C \circ \rho$ is an endomorphism of $\B_n$.

By the definition of a 2-fold cabling map, we have that the post-composition of $\Pi_e^C \circ \rho$ by the projection $\B_n \to S_n$ is standard.  In particular $\Pi_e^C \circ \rho$ does not have cyclic image.   It then follows from Theorem~\ref{thm:castel} that $\Pi_e^C \circ \rho$ is equivalent to the identity.  In other words, there is an $A \in \Aut \B_n$ and a $t \in \in Z(\B_n)$ so that
\[
A \circ \left( \Pi_e^C \circ \rho \right)^t
\]
equals the identity map $\B_n \to \B_n$.  

Say that $A$ is induced by the (possibly orientation-reversing) homeomorphism $f$.  There is a homeomorphism $\tilde f$ of $\D_{2n}$ that fixes the components of $C$ and induces $f$ under the operation of collapsing the disks bounded by the components of $C$ (this is analogous to the map $\Pi_e^C$, but allowing for orientation-reversing maps).  Conjugation by $\tilde f$ induces an automorphism $\tilde A \in \Aut \B_{2n}$.  This $\tilde A$ has the property that $A\circ \Pi_e^C = \Pi_e^C \circ \tilde{A}$.  Let $u = z^k \in \B_{2n}$ be the unique such element so that $\Pi_e^C(u)=A(t)$.

We claim that
\[
A \circ \left(\Pi_e^C \circ \rho\right)^t = \Pi_e^C \circ \left( \tilde A \circ \rho^{\tilde A^{-1}(u)}\right).
\]
Since the former is the identity, the claim implies that the latter is the identity, which gives the second statement.  

By evaluating on the generators $\sigma_i$, we have that
\[
A \circ \left(\Pi_e^C \circ \rho\right)^t  = (A\circ \Pi_e^C\circ \rho)^{A(t)} = ( \Pi_e^C \circ \tilde{A}\circ \rho)^{A(t)}
\]
The latter is a homomorphism because the first one is.   By evaluating on the generators $\sigma_i$, we know that
\[
( \Pi_e^C \circ \tilde{A}\circ \rho)^{A(t)} = \Pi_e^C \circ ( \tilde{A}\circ \rho)^u.
\]
The map $( \tilde{A}\circ \rho)^u$ is a homomorphism since it is a central transvection of a homomorphism.  Similarly by evaluating on the generators, we also have 
\[( \tilde{A}\circ \rho)^u= \tilde{A} \circ \rho^{\tilde{A}^{-1}(u)}
\]
Combining the above equalities gives the claim.
\end{proof}

\begin{proof}[Proof of Proposition~\ref{prop:cabling}]

By Lemma~\ref{lem:cable} we may assume that the cabling multicurve for $\rho$ is the standard multicurve $C$, and that $\rho$ is a section of $\Pi_e^C$.  As per Section~\ref{sec:pkg}, the kernel of $\Pi_e^C : \Stab_{\B_{2n}}(C) \to \B_n$ is isomorphic to $\B_2 \times \cdots \times \B_2 \cong \Z^n$.  With respect to the semi-direct product decomposition for $\Stab_{\B_{2n}}(C)$ from Lemma~\ref{lem:pkg2} we can thus write the standard $k$-twist cabling map $\rho_k$ as 
\[
\rho_k(\sigma_i) = (\sigma_i,(0,\dots,0,k,0,\dots,0))
\]
where $k$ lies in the $i$th entry.  Similarly, since $\rho(\B_n)$ also lies in $\Stab_{\B_{2n}}(C)$, it follows from Lemma~\ref{lem:cable}(2) that we may describe each $\rho(\sigma_i)$ as
\[
\rho(\sigma_i) = (\sigma_i,(k_{i,1},\dots,k_{i,n}))
\]
for some $(k_{i,1},\dots,k_{i,n}) \in \Z^n$.  In the statement of the following claim we denote $k_{1,4}$ by $a$.

\p{Claim} Whenever $j \notin \{i,i+1\}$ we have $k_{i,j} = a$.  In other words, for all $i$ we have
\[
\rho(\sigma_i) = (a,\dots,a,k_{i,i},k_{i,i+1},a, \dots,a).
\]

\p{Proof of claim} We first consider the commuting relations and then the braid relations.  If $|i-j| > 1$, then $\sigma_i$ commutes with $\sigma_j$.  Thus
\[
\rho(\sigma_i)\rho(\sigma_j) = \rho(\sigma_j)\rho(\sigma_i).
\]
Using the formulas for $\rho(\sigma_i)$ and $\rho(\sigma_j)$ above and using the definition of the semi-direct product structure on $\B_n \ltimes \Z^n$, we conclude that 
\begin{align*}
k_{i,i}+k_{j,i+1}&=k_{j,i}+k_{i,i} \ \ \text{ and } \\
k_{j,j+1}+k_{i,j}&=k_{i,j+1}+k_{j,j+1}. 
\end{align*}
(The commuting relation gives $n$ relations between the $k_{i,p}$ and the $k_{j,q}$, but $n-4$ are trivial and of the remaining four, two of the relations differ from the other two by interchanging $i$ and $j$.) 

In the same way, we obtain equalities from the braid relations $\sigma_i\sigma_{i+1}\sigma_i=\sigma_{i+1}\sigma_{i}\sigma_{i+1}$.  For $\ell \notin \{i,i+1,i+2\}$ we obtain from the $\ell$th cable the equation
\[
2k_{i,\ell}+ k_{i+1,\ell} = k_{i,\ell} + 2k_{i+1,\ell},
\]
and we conclude that 
\[
k_{i,\ell}=k_{i+1,\ell}\ \  \text{ for } \ \ \ell \notin \{i,i+1,i+2\}.
\]
From the $i$th and $(i+1)$st cables we obtain the equalities
\begin{align*}
k_{i,i}+k_{i+1,i+1}+k_{i,i+2} &= k_{i+1,i}+k_{i,i}+k_{i+1,i+1} \ \ \text{ and}\\
k_{i,i+1}+k_{i+1,i}+k_{i,i}  &= k_{i+1,i+1}+k_{i,i+2}+k_{i+1,i+2}.
\end{align*}
Thus
\begin{align*}
k_{i,i+2} &= k_{i+1,i}\qquad\qquad\qquad  \text{ and}\\
k_{i,i+1}+k_{i,i} &= k_{i+1,i+1}+k_{i+1,i+2}.
\end{align*}

From the four sets of equations
\begin{align*}
\qquad &k_{j,i+1}=k_{j,i} \ \ \ \ \, \text{ for } |i-j|>1, 
k_{i,i+2}&=k_{i+1,i} \text{ and }& k_{i,\ell}=k_{i+1,\ell} \ \ \text{ for } \ell \notin \{i,i+1,i+2\},
\end{align*}
the claim now follows.

\bigskip

We now use the claim to complete the proof.  Let $t = (\sigma_1\sigma_3\cdots \sigma_{2n-1})^{-a}$.  This element lies in the center of $\Stab_{\B_{2n}}(C)$ and hence centralizes $\rho(\B_n)$.  The transvection of $\rho$ by $t$ has the effect of subtracting $a$ from each coordinate of the second factor of each $\rho(\sigma_i)$.  Thus
\[
\rho^t(\sigma_i) = (\sigma_i,(0,\dots,0,k_{ii}-a , k_{i,i+1}-a,0,\dots,0)).
\]
By basic linear algebra, the system of equations 
\[
\{ x_{i+1} - x_i = a- k_{i,i+1}\},
\]
in the variables $x_1,\dots, x_n$ has infinitely many integer solutions.  We fix one such vector of solutions $(e_1,\dots,e_n)$.

Let $A$ denote the inner automorphism of $\B_{2n}$ corresponding to conjugation by
\[
h = \sigma_1^{e_1}\sigma_3^{e_2}\cdots \sigma_{2n-1}^{e_n}.
\]
Then
\[
A \circ \rho^t(\sigma_i) = (\sigma_i,(0,\dots,0, a_i,b_i,0,\dots,0))
\]
where 
\begin{align*}
a_i &= e_i-e_{i+1}+k_{i,i}-a \qquad \text{ and}\\
b_i &= e_{i+1}-e_i+k_{i,i+1}-a.
\end{align*}
By our choice of the $e_i$, we have that the $b_i$ are all zero and for each $i$ we have
\begin{align*}
a_i &= e_i-e_{i+1}+k_{i,i}-a \\
&= (k_{i,i+1}-a) + k_{i,i} - a \\
&= (k_{i,i+1}+ k_{i,i}) - 2a.
\end{align*}
Since $k_{i,i+1}+k_{i,i}=k_{i+1,i+1}+k_{i+1,i+2}$, the latter is independent of $i$, as desired.
\end{proof}



\section{Proof of the main theorem}
\label{sec:proof}

Theorem~\ref{thm:main} says that every homomorphism $\B_n \to \B_{2n}$ is equivalent to exactly one standard homomorphism.  In this section we first prove the uniqueness statement (Proposition~\ref{prop:inequ}).  Then we complete the proof of Theorem~\ref{thm:main} using both Theorems~\ref{thm:castel} and~\ref{thm:castel2}.  

\begin{proposition}
\label{prop:inequ}
Let $n \geq 5$.  The standard homomorphisms $\B_n\rightarrow \B_{2n}$ are pairwise inequivalent.
\end{proposition}

\begin{proof}

If two standard homomorphisms $\B_n\rightarrow \B_{2n}$ are equivalent, then their restrictions to $\B_n'$ are almost-conjugate.  Indeed, this follows from three facts: the fact that transvections of standard homomorphisms are central (Corollary~\ref{cor:central}), the fact that central transvections are given by the formula $\rho^t(g) = \rho(g)t^{L(g)}$ (where $L$ is the abelianization), and the fact that the restriction $L|\B_n'$ is trivial.  

By the previous paragraph, it suffices to show that the restrictions to $\B_n'$ of the standard homomorphisms are pairwise not almost-conjugate.  To this end, we consider the image of $\sigma_1\sigma_3^{-1}$ under each standard homomorphism. 
\begin{enumerate}
\item trivial map: \emph{identity}
\item inclusion: $\sigma_1\sigma_3^{-1}$
\item diagonal inclusion: $\sigma_1\sigma_3^{-1}\sigma_{n+1}\sigma_{n+3}^{-1}$
\item flip diagonal inclusion: $\sigma_1\sigma_3^{-1}\sigma_{n+1}^{-1}\sigma_{n+3}$
\item $k$-twist cabling map:  $(\sigma_{2}\sigma_{1}\sigma_{3}\sigma_{2}\sigma_{1}^k)(\sigma_{6}\sigma_{5}\sigma_{7}\sigma_{6}\sigma_{5}^k)^{-1}$
\end{enumerate}
Let $c_{1234}$ and $c_{5678}$ be the round curves in $\D_{2n}$ surrounding the first four and next four marked points, respectively.  The square of the last element is the multitwist
\[
(\sigma_1^k \sigma_3^k T_{c_{1234}} )(\sigma_5^k \sigma_7^k T_{c_{5678}} )^{-1}.
\]
For distinct $k$, the corresponding multitwists are do not differ by an automorphism of $\B_{2n}$ (cf. \cite[Lemma 1]{lantern}), and so the $k$-twist cabling maps are pairwise inequivalent.  

The canonical reduction systems for the above images are as follows.
\begin{enumerate}
\item trivial map: $\empty$
\item inclusion: $\{c_1,c_3\}$
\item diagonal inclusion: $\{c_1,c_3,c_{n+1},c_{n+3}\}$
\item flip diagonal inclusion: $\{c_1,c_3,c_{n+1},c_{n+3}\}$
\item $k$-twist cabling map:  $\{c_1,c_3,c_{1234},c_5,c_7,c_{5678}\}$
\end{enumerate}
Simply by counting the numbers of curves in these canonical reduction systems, we see that the only two standard homomorphisms that can possibly be equivalent are the diagonal inclusion and the flip diagonal inclusion.  So it remains to show that their restrictions to $\B_n'$ are not almost-conjugate.  For this we consider the images of $\sigma_1^{L(z)} z^{-1}$.  Let $d_0$ and $d_1$ be the round curves surrounding the first $n$ and last $n$ marked points, as in the previous subsection.  The images are as follows.
\begin{enumerate}
\item diagonal inclusion: $(\sigma_1\sigma_{n+1})^{L(z)} (T_{d_0}T_{d_1})^{-1}$
\item flip diagonal inclusion: $(\sigma_1\sigma_{n+1}^{-1})^{L(z)} T_{d_0}^{-1}T_{d_1}$
\end{enumerate}
These do not differ by an automorphism of $\B_{2n}$, and so the proof is complete.
\end{proof}

We now prove the main theorem.

\begin{proof}[Proof of Theorem~\ref{thm:main}]

As in the statement, we fix some $n \geq 5$.  By Proposition~\ref{prop:inequ}, the standard homomorphisms $\B_n\rightarrow \B_{2n}$ are pairwise inequivalent. Also, by Corollary~\ref{cor:central}, a trasvection of a standard homomorphism is central, and so the second statement of the theorem follows from the first.  It remains to show that any homomorphism $\B_n\rightarrow \B_{2n}$ is equivalent a standard homomorphism.   

To this end, we will prove by strong induction on $m$ the following statement: for $m \leq 2n$, every homomorphism $\B_n \to \B_m$ is equivalent to a standard homomorphism, and in particular if $m < 2n$ then it is centrally equivalent to either the trivial map or to the inclusion map.  The base case is $m=n$, which is Theorem~\ref{thm:castel}.  Therefore, we may assume that $n < m \leq 2n$.  

Denote the canonical reduction system of $\rho(z)$ by $M$.  It follows from Corollary~\ref{cor:numthy} and Proposition~\ref{prop:torsiontopA} that if $M$ is empty then the induced map $\bar \rho$ has cyclic image, and then from Lemma~\ref{lem:cycliccyclic} that $\rho$ has cyclic image (hence is equivalent to the trivial map). Thus, we may henceforth assume that $M$ is non-empty. 

As in the proof of Theorem~\ref{thm:castel}, the map $\rho$ induces an action of $\B_n$ on the set of components of $M$.  We may write $M$ as a union of multicurves $M_p$, where each component of $M_p$ is a curve surrounding exactly $p$ marked points.  The action of $\B_n$ further restricts to an action on each $M_p$.  Each $M_p$ has at most $m/p$ components.  Since $m \leq 2n$ and $p \geq 2$, each $M_p$ has at most $n$ components.  

For each $p$, the action of $\B_n$ on the set of components of $M_p$ factors through a homomorphism
\[
\Pi_p : \B_n \to \B_k
\]
where $k=|M_p|$.  The map $\Pi_p$ is the composition of two maps.  The first is the map
\[
\Pi_e^{M_p} : \B_n \to \B_{n-kp+k}
\]
from Section~\ref{sec:pkg} and the second is the map $\B_{n-kp+k} \to \B_k$ obtained by forgetting the marked points not coming from $M_p$.
 
Since $k < m \leq 2n$, it follows from induction that $\Pi_p$ is equivalent to either the trivial map or the identity map.  There are two possibilities:
\begin{enumerate}
\item there is a $p$ so that $\Pi_p$ is equivalent to the identity map, or
\item all of the $\Pi_p$ are equivalent to the trivial map.
\end{enumerate}
In the first case it must be that $p=2$ and the action of $\rho(\B_n)$ on the set of components of $M_2$ is standard.  By Proposition~\ref{prop:cabling}, $\rho$ is equivalent to a $k$-twist cabling map. 

We may assume henceforth that we are in the second case.  In this case, it must be that the action of $\rho(\B_n)$ on the set of components of each $M_p$, hence on the set of components of $M$, is cyclic.  

Let $P$ be the largest number so that $M_P$ is nonempty.  If $P < n$, then we may use a similar argument to the one used in Case 3 of the proof of Theorem~\ref{thm:castel} to show that $\rho$ has cyclic image.  Indeed, since $\Pi_p$ is equivalent to the trivial map, it has cyclic image.  Thus, the image of $\B_n'$ under $\Pi_p$ is trivial.  As in the proof of Theorem~\ref{thm:castel}, a theorem of Lin gives that $\rho(\B_n')$ is trivial, and hence $\rho(\B_n)$ has cyclic image.  

So we may henceforth assume that $P \geq n$.  We now complete the proof in three separate cases:
\begin{enumerate}
\item $M_P$ has one component.
\item $M_P$ has two components, both fixed by $\rho$.
\item $M_P$ has two components, interchanged by the action of $\rho$.
\end{enumerate}
In both of the latter cases, it must be that $m=2n$.

\bigskip

\noindent \emph{Case 1: $M_P$ has a single component.}  The argument here is based on the argument for Case~3 in the proof of Theorem~\ref{thm:castel}.  The main difference is that we must use our strong inductive hypothesis instead of Lin's theorem.  

Since $M_P$ has a single component, we may apply the interior/exterior decomposition from Section~\ref{sec:pkg}.  Let $\rho_e^{M_P}$ and $\rho_i^{M_P}$ be the corresponding interior and exterior components of $\rho$.  If $k=m-P+1$, these maps have target $\B_k$ and $\B_p$, respectively.  We observe that $k \leq n$ since $P \geq n$ and $m \leq 2n$.  

We claim that $\rho_e^{M_P} : \B_n \to \B_{k}$ has cyclic image.  If $k < n$, then we may conclude as above that $\rho_e^{M_P}$ has cyclic image.  If $k=n$ then we know by Theorem~\ref{thm:castel} that $\rho_e^{M_P}$ has cyclic image or is equivalent to the identity map.  But since $\rho_e^{M_P}(\B_n)$ lies in $\B_{k-1,1} \subsetneq \B_{k} = \B_n$, it must be that $\rho_e$ has cyclic image, as desired.

By strong induction, $\rho_i^{M_P}$ is equivalent to the standard inclusion $\B_n \to \B_P$.  Since $\rho_e^{M_P}$ has cyclic image, we may modify $\rho$ by a transvection so that $\rho_e^{M_P}$ is the trivial map.  It follows from Lemma~\ref{lem:pkg} that $\rho$ is equal to the post-composition of $\rho_i^{M_P}$ with the inclusion $\B_P \to \B_m$.  In particular, $\rho$ is equivalent to the standard inclusion.

\bigskip

\noindent \emph{Case 2: $M_P$ has two components, both fixed by $\rho$.} Since the components of $M_P$ are both fixed by $\rho$, we may again apply Lemma~\ref{lem:pkg}.  

The exterior map $\rho_e^{M_P}$ has cyclic image, since its image is a subgroup of the cyclic group $\PB_2$.  Therefore, after modifying $\rho$ by a transvection, we may assume that $\rho_e^{M_P}$ is trivial.

The interior component of $\rho_i^{M_P}$ is
\[
\rho_i : \B_n \to \B_n \times \B_n.
\]
It follows from Theorem~\ref{thm:castel} (or induction) and the fact that each element of $\Aut(\B_n)$ is induced by a homeomorphism of $\D_n$ that, up to almost-conjugation, each of the two components of $\rho_i$ is a transvection of either the identity map, the inversion map, or the trivial map, and further that at most one of the factors is the inversion map. It now follows from Lemma~\ref{lem:pkg} that $\rho$ is equivalent to either the trivial homomorphism, the diagonal inclusion, or the flip-diagonal inclusion.

\bigskip

\noindent \emph{Case 3: $M_P$ has two components, interchanged by the action of $\rho$.}  Let us denote the components of $M_P$ by $d_0$ and $d_1$.  Up to conjugating $\rho$, we may assume that these are the standard curves surrounding the first and last $n$ marked points, respectively.  For $1 \leq i \leq n-1$, let $c_i$ and $c_i'$ denote the standard curves in $\D_{2n}$ surrounding the marked points $\{i,i+1\}$ and $\{n+i,n+i+1\}$, respectively (so $\sigma_i$ and $\sigma_{n+i}$ are the half-twists about $c_i$ and $c_i'$, respectively).  

In the present case Lemma~\ref{lem:pkg2} gives the semi-direct product decomposition
\[
\Stab_{\B_{2n}}(M_P) \cong \B_2 \ltimes (\B_n \times \B_n).
\]
where the generator $\sigma_1$ for $\B_2 \cong \Z$ acts on $\B_n \times \B_n$ by interchanging the factors.  In terms of this decomposition, the assumption that $\rho(\B_n)$ acts nontrivially on the components of $M_P$ translates to the fact that $\rho(\B_n)$ projects to a subgroup of $\B_2$ that is not contained in $\PB_2$.  In particular, there is an odd $\ell$ so that each $\rho(\sigma_i)$ projects to $\sigma_1^\ell \in \B_2$.   

Since $\rho(\B_n)$ maps to $\Stab_{\B_{2n}}(M_P)$ by assumption, we may write elements of $\rho(\B_n)$ in terms of the semi-direct product decomposition.  So elements of $\rho(\B_n)$ will be written as $(\sigma_1^\ell,(\alpha,\beta))$, where $k$ is an integer, and $\alpha$ and $\beta$ are elements of $\B_n$.

Consider the restriction $\rho^2 = \rho|\B_n^2$ (this is not to be confused with $\rho^{\circ 2} = \rho \circ \rho$).  The image of $\rho^2$ lies in
\[
\Fix_{\B_{2n}}(M_P) \cong \B_2^2 \times (\B_n \times \B_n) = \PB_2 \times (\B_n \times \B_n).
\]
This group has index 2 in $\Stab_{\B_{2n}}(M_P)$.  As in Section~\ref{sec:pkg}, we may post-compose $\rho^2$ with the projections to $\PB_2$ and $\B_n \times \B_n$, and we denote the resulting homomorphisms by $\rho_e^2$ and $\rho_i^2$.  We may further post-compose $\rho_i^2$ with the projections to the two factors of $\B_n \times \B_n$ in order to obtain homomorphisms $\rho_0^2$ and $\rho_1^2$.  

By Theorem~\ref{thm:castel2} we may assume (up to almost-conjugation) that $\rho_0^2$ is a central transvection of either the trivial map or the inclusion map, and that $\rho_1^2$ is a central transvection of the trivial map, the inclusion map, or the (restriction of the) inversion map; we refer to the latter as the inverse inclusion map.

The braid $T_{d_0}T_{d_1}$ centralizes the image of $\rho$.  We may modify $\rho$ by the transvection associated to a power of $T_{d_0}T_{d_1}$ so that $\rho_0^2$ is either the trivial map or inclusion map and $\rho_1'$ is a transvection of the trivial map, the inclusion map, or the restriction of the inversion map.  

It follows from the semi-direct product decomposition for $\Stab_{\B_{2n}}(M_P)$ and the fact that $\rho(\sigma_1)$ projects to an element of $\B_2 \setminus \PB_2$ that $\rho_1^2$ is trivial if and only if $\rho_0^2$ is.  If both are trivial, then by the direct product decomposition for $\Fix_{\B_{2n}}(M_P)$ it follows that $\rho^2$ has cyclic image.  By Lemma~\ref{lem:bn2abel} it further follows that $\rho$ has cyclic image.  Thus, we may henceforth assume that $\rho_0^2$ is inclusion and $\rho_1'$ is a transvection of either the inclusion map or the inverse inclusion map.  

We claim now that $\rho_1'$ is the inclusion map.  Since $\rho_1'$ is a transvection of either the inclusion map or the inverse inclusion map, we have that $\rho(\sigma_1^2) = \rho'(\sigma_1^2)$ is of the form
\[
(\sigma_1^{2\ell},(\sigma_1^2,\sigma_1^{2\epsilon}z^k))
\]
where $\epsilon \in \{\pm 1\}$.  Since $\rho(\sigma_1)$ commutes with $\rho(\sigma_1^2)$, the former preserves the canonical reduction system of the latter.  Since $c_1$ and $c_1'$ are the only curves in the canonical reduction system of $\rho(\sigma_1^2)$ that surround exactly two marked points, and since $\rho(\sigma_1)$ permutes the two components of $M_P$, it follows that $\rho(\sigma_1)$ interchanges $c_1$ and $c_1'$.   Again using the fact that $\rho(\sigma_1)$ commutes with $\rho(\sigma_1^2)$, it follows that $\epsilon = 1$.  By the same reasoning applied to $d_0$ and $d_1$, we have $k=0$, whence the claim.

We have shown that $\rho'$ is the restriction of the diagonal inclusion map.  We will use this to show that $\rho$ is a transvection of the diagonal inclusion map.

We claim that for $1 \leq i \leq n-1$ we have
\[
\rho(\sigma_i) = (\sigma_1^\ell,(\sigma_i,\sigma_i))
\]
The canonical reduction system for $\rho(\sigma_i)$ is the same as that for $\rho(\sigma_i^2)$, namely, $\{c_i,c_{n+i},d_0,d_1\}$.  Since $n \geq 5$ there is a $\sigma_j$ that commutes with $\sigma_i$ and has disjoint canonical reduction system $\{c_i,c_{n+i},d_0,d_1\}$.  Since $\rho(\sigma_i)$ preserves the latter, it follows that $\rho(\sigma_i)$ is of the form $(\sigma_1^p,(\sigma_i^q,\sigma_i^r))$ (it cannot be pseudo-Anosov or a nontrivial rotation on the region between $c_i$ and $d_0$ or the region between $c_{n+i}$ and $d_1$ since it preserves $c_j$ and $c_{n+j}$).  We can then conclude that $p$, $q$, and $r$ are $\ell$, 1, and 1 since the square is $(\sigma_1^{2\ell},(\sigma_1,\sigma_1^2))$.  This completes the proof of the claim.

\medskip

We now claim that the transvection of $\rho$ by
\[
t = (\sigma_1^{-\ell},(id,id))
\]
is diagonal inclusion.  First of all, the braid $t$ lies in the centralizer of $\rho(\B_n)$.  Indeed, for $1 \leq i \leq n-1$ we have that $t\rho(\sigma_i)$ and $\rho(\sigma_i)t$ are both equal to 
\[
(1,(\sigma_i,\sigma_i)).
\]
This calculation also implies the claim, and hence completes the proof of the theorem.
\end{proof}


\section{Proof of Corollary~\ref{cor:bnp}}
\label{sec:cor}

In this section we prove Corollary~\ref{cor:bnp}, which states that if $n \geq 7$ and $\rho : \B_n' \to \B_{2n-5}$ is a nontrivial homomorphism, then $\rho$ is almost-conjugate to the inclusion map.  The argument given here is a simplification of an argument suggested by an anonymous referee for the paper by the second and third authors \cite{bnprime}.

\begin{proof}[Proof of Corollary~\ref{cor:bnp}]

There is an inclusion $\nu : \B_{n-2} \to \B_n'$ given by $\nu(\sigma_i) = \sigma_i \sigma_{n-1}^{-1}$ for $i=1,\dots,n-3$.  By Theorem~\ref{thm:main}, the composition
\[
\B_{n-2} \stackrel{\nu}{\to} \B_n' \stackrel{\rho}{\to} \B_{2n-5}
\]
is centrally equivalent to either the trivial map or the inclusion map.  By post-composing with an automorphism of $\B_{2n-5}$ we may assume that the composition is a central transvection of either the trivial map or the inclusion map.    We will show in the first case that $\rho$ is trivial and that in the second case $\rho$ is a transvection of the inclusion map.    

Assume first that the composition is a central transvection of the trivial map.  Such a map has cyclic image, and so $\rho \circ \nu(\sigma_i)$ is independent of $i$.  It follows that $\rho \circ \nu(\sigma_1\sigma_2^{-1}) = \rho(\sigma_1\sigma_2^{-1})$ is trivial.  Since the normal closure of $\sigma_1\sigma_2^{-1}$ in $\B_n'$ is $\B_n'$ (Lemma~\ref{lem:bn2nc}), the map $\rho$ is trivial, as desired.

We now assume that the composition is a central transvection of the inclusion map.  This means that there is some $\beta$ in the centralizer of the image of $\rho \circ \nu$ so that $\rho \circ \nu(\sigma_i) = \sigma_i \beta$ for $i \in \{1,\dots,n-3\}$.  As in the previous paragraph, it follows that $\rho \circ \nu(\sigma_1\sigma_j^{-1}) = \rho(\sigma_1\sigma_j^{-1})$ is equal to $\sigma_1\sigma_j^{-1}$ for $j \in \{2,\dots,n-3\}$.  As in the proof of Lemma~\ref{lem:emptyCRS2}, the group $\B_n'$ is generated by $\sigma_1\sigma_j^{-1}$ for $j \in \{2,\dots,n-1\}$, and so we focus our attention on the $\rho$-images of $\sigma_1\sigma_j^{-1}$ for $j = n-2,n-1$.  It is enough to show, after possibly post-composing $\rho$ by an inner automorphism of $\B_{2n-5}$, that $\rho(\sigma_1\sigma_j^{-1}) = \sigma_1\sigma_j^{-1}$ for $j \in \{2,\dots,n-1\}$.

We claim that $\rho(\sigma_1\sigma_{n-1}^{-1})$ equals $\sigma_1 H_{d}^{-1}$ for some curve $d$ surrounding two marked points.  Since $\sigma_1\sigma_{n-1}^{-1}$ is conjugate to $\sigma_1\sigma_3^{-1}$ inside $B_n'$ and since we have $\rho(\sigma_1\sigma_3^{-1})=\sigma_1\sigma_3^{-1}$, we know that $\rho(\sigma_1\sigma_{n-1}^{-1})=H_{a_1} H_{b_1}^{-1}$ where $a_1$ and $b_1$ are disjoint curves surrounding two marked points each. For the same reason,  we have for any for $j \in \{3,\dots,n-3\}$ that $\rho(\sigma_j\sigma_{n-1}^{-1})=H_{a_j} H_{b_j}^{-1}$ where $a_j$ and $b_j$ are disjoint curves surrounding two marked points each.  We have
\[
\sigma_1\sigma_j^{-1}=\rho(\sigma_1\sigma_j^{-1})= \rho(\sigma_1\sigma_{n-1}^{-1})\rho(\sigma_j\sigma_{n-1}^{-1})^{-1}=H_{a_1} H_{b_1}^{-1}H_{a_j}^{-1} H_{b_j}.
\]
Since $\sigma_1\sigma_{n-1}^{-1}$ commutes with $\sigma_j\sigma_{n-1}^{-1}$, it must be that the curves $a_1$, $b_1$, $a_j$, and $b_j$ have trivial intersection pairwise.  In particular, both $\sigma_1\sigma_j^{-1}$ and $H_{a_1} H_{b_1}^{-1}H_{a_j}^{-1} H_{b_j}$ are multitwists, that is, each is the product of powers of half-twists about pairwise disjoint curves.  If multitwists are equal in $\B_n$, then (up to reordering terms) they must be formally the same (they use the same curves and the same power on each curve); see \cite[Lemma 1]{lantern}.  Thus in the expression $H_{a_1} H_{b_1}^{-1}H_{a_j}^{-1} H_{b_j}$ it must be that exactly two twists cancel, and the remaining twists are $\sigma_1$ and $\sigma_j^{-1}$.  As $a_1 \neq b_1$ and $a_j \neq b_j$ there are two possibilities.  The first is that $a_1=a_j$, in which case $b_j=c_1$ and $b_1=c_j$.  The second possibility is that $b_1=b_j$, in which case $a_1=c_1$ and $a_j=c_j$.  Under the first possibility, we have that $\rho(\sigma_1\sigma_{n-1}^{-1})$ is equal to $H_{a_1}H_{c_j}^{-1}$.  But taking $j$ to be either 3 or 4 in the last expression, we get two different braids, which is a contradiction.  So we must have the second possibility, which means that $\rho(\sigma_1\sigma_{n-1}^{-1})$ is equal to $H_{c_1}H_{b_1}^{-1} = \sigma_1 H_{b_1}^{-1}$, as desired.

We next claim that the curve $d$ has the following intersection numbers:
\[
i(d,c_j) = 0\ \  \text{ for } \ \  1 \leq j \leq n-3
\]
Fix one such $j$.  We have that $\sigma_j\sigma_{n-1}^{-1}$ is conjugate in $\B_n'$ to $\sigma_1\sigma_{3}^{-1}$.  The $\rho$-image of the latter is $\sigma_1\sigma_{3}^{-1}$ and so the $\rho$-image of $\sigma_j\sigma_{n-1}^{-1}$ must be a difference of commuting half-twists.  We compute this image as follows:
\[
\rho(\sigma_j\sigma_{n-1}^{-1}) = \rho(\sigma_j\sigma_1^{-1}\sigma_1\sigma_{n-1}^{-1}) = \rho(\sigma_j\sigma_1^{-1})\rho(\sigma_1\sigma_{n-1}^{-1}) = \sigma_j\sigma_1^{-1}\sigma_1H_d^{-1} = \sigma_jH_d^{-1}.
\]
By (a version of) the Thurston construction, the difference of two non-commuting half-twists is a partial pseudo-Anosov braid, and hence is not conjugate to $\sigma_1\sigma_{3}^{-1}$; see Step 3 of Case 4 of the proof of Theorem~\ref{thm:main} in the paper by the second and third authors \cite{bnprime}.  The claim follows. 

By the previous claim, we may assume, up to conjugacy of $\rho$, that $d = c_{n-1}$ and that $\rho|B_{n-2}$ is still the standard inclusion.  

We next claim that $\rho(\sigma_1\sigma_{n-2}^{-1})$ is equal to $\sigma_1H_{e}^{-1}$ for some curve $e$ surrounding two marked points. For $n \geq 8$, we can use the same argument as the previous paragraph, replacing $n-3$ with $n-4$ (in this case $j=3$ and $j=4$ both lie in $\{3,\dots,n-4\}$).  For $n=7$ we require a different argument to rule out the analogue of the first possibility in the previous claim, namely, that $\rho(\sigma_1\sigma_{5}^{-1})$ is equal to $H_{a_1}\sigma_3^{-1}$ for some curve $a_1$ that surrounds two marked points and is disjoint from $c_3$.  By the previous claim, we have
\[
\rho(\sigma_5\sigma_6^{-1})=\rho((\sigma_1\sigma_5^{-1})^{-1}(\sigma_1\sigma_6^{-1}))=H_{a_1}^{-1}\sigma_3\sigma_1\sigma_6^{-1} = H_{a_1}^{-1}\sigma_6^{-1}\sigma_3\sigma_1.
\] 
Since $\sigma_5\sigma_6^{-1}$ is conjugate to $\sigma_1\sigma_2^{-1}$ in $B_n'$ and since $\rho(\sigma_1\sigma_2^{-1}) = \sigma_1\sigma_2^{-1}$, it must be that the action of $\rho(\sigma_5\sigma_6^{-1})$ on the set of marked points in $\D_{2n-5}$ is a 3-cycle.  On the other hand, the action of $H_{a_1}^{-1}\sigma_6^{-1}\sigma_3\sigma_1$ on the marked points is the product of three disjoint 2-cycles and one additional 2-cycle (the latter corresponding to $a_1$).  Such a permutation cannot be a 3-cycle, and the claim is proved.  

Similar to the claim regarding the $i(d,c_j)$, we claim that
\begin{align*}
i(e,c_j) = \begin{cases} 0 & 1 \leq j \leq n-4 \\
2 & j \in \{n-3,n-1\} \end{cases}
\end{align*}
For example, Since $\sigma_1\sigma_{n-2}^{-1}$ and $\sigma_1\sigma_{n-1}^{-1}$ satisfy the braid relation, the same is true for their images under $\rho$.  From this it follows that $H_e$ and $\sigma_{n-1}$ satisfy the braid relation.  As in Case 2 of the proof of Theorem~\ref{thm:castel}, two half-twists satisfy the braid relation if and only if the corresponding curves intersect in exactly two points.  Thus, $i(e,c_{n-1})=2$, as desired. 

It follows from the last claim that, after post-composing $\rho$ with an inner automorphism of $\B_{2n-5}$, we have $\rho(\sigma_1\sigma_j^{-1}) = \sigma_1\sigma_j^{-1}$ for $j \in \{2,\dots,n-1\}$, as desired.
\end{proof}

\bibliographystyle{plain}
\bibliography{bn2bn}

\end{document}